\documentclass[11pt]{article}

\usepackage{latexsym,mathrsfs}
\usepackage{amsmath,amssymb}
\usepackage{amsthm,enumerate,verbatim}
\usepackage{amsfonts}
\usepackage{graphicx}
\usepackage{algorithm}
\usepackage{algorithmic}
\usepackage{url}

\newtheorem{lemma}{Lemma}

\newtheorem{theorem}{Theorem}

\newtheorem{corollary}{Corollary}

\newtheorem*{definition*}{Definition}

\newtheorem{remark}{Remark}

\numberwithin{equation}{section}
\numberwithin{table}{section}
\numberwithin{figure}{section}
\def \R{{\mathbb R}}
\def \C{{\mathbb C}}
\def\real{\mathop{\mathrm{Re}}}

\newcommand {\mat}  [1] {\left[\begin{array}{#1}}
\newcommand {\rix}      {\end{array}\right]}

\def\real{\mathop{\mathrm{Re}}}

\title{
Minimal-norm static feedbacks using dissipative Hamiltonian matrices
}

\author{Nicolas Gillis\thanks{Department of Mathematics and Operational Research, University of Mons, Rue de Houdain 9, 7000 Mons, Belgium. Email: nicolas.gillis@umons.ac.be. N.~Gillis acknowledges the support by the Fonds de la Recherche Scientifique - FNRS and the Fonds Wetenschappelijk Onderzoek - Vlanderen (FWO) under EOS Project no O005318F-RG47, and by the European Research Council (ERC starting grant no 679515).}
\qquad Punit Sharma\thanks{Department of Mathematics, Indian Institute of Technology Delhi, Hauz Khas, New Delhi-110016, India.
Email: punit.sharma@maths.iitd.ac.in. P. Sharma acknowledges the support of the DST-Inspire Faculty Award (MI01807-G) by Government of India and Institute SEED Grant (NPN5R) by IIT Delhi.}}

\begin{document}

\maketitle

\begin{abstract}
In this paper, we characterize the set of static-state feedbacks that stabilize a given
continuous linear-time invariant system pair using dissipative Hamiltonian matrices. This characterization results
in a parametrization of feedbacks in terms of skew-symmetric and symmetric positive semidefinite matrices, and leads to a semidefinite program that computes a static-state stabilizing feedback. 
This characterization also allows us to propose an algorithm that computes minimal-norm static feedbacks. 
 The theoretical results extend to the static-output feedback (SOF) problem, and we also propose an algorithm to tackle this problem. 
 We illustrate the effectiveness of our algorithm compared to state-of-the-art methods for the SOF problem on numerous numerical examples from the COMPLeIB library. 
\end{abstract}


\section{Introduction}\label{sec:intero}

Consider a continuous linear-time invariant (LTI) system in the form
\begin{eqnarray*}
\dot{x}(t)&=&A x(t)+Bu(t), \\
y(t)&=&Cx(t),
\end{eqnarray*}
where, for all $t\in \R$, $x(t) \in \R^n$ is the state space, $u(t) \in \R^{m}$
is the control input, $A\in \R^{n,n}$, $B\in \R^{n,m}$, and $C\in \R^{p,n}$.
The notion of stabilizing the system pair $(A,B)$
using feedback controllers is a fundamental one, and is referred to as the static-state feedback (SSF) problem. 
It requires to find $K \in \R^{m,n}$ such that
$A-BK$ is stable, that is, all eigenvalues of the matrix $A-BK$ are in the left half of the complex plane and those on the imaginary axis are semisimple; see for example~\cite{AstM10,AndBJ75}. 
The first goal of this paper is to solve the SSF problem; this can be divided into two parts: 
\begin{enumerate}
\item Feasibility. 
Check the existence of a feedback matrix $K$ such that $A-BK$ is stable.

\item Optimization. 
If the problem is feasible, minimize the  norm of the feedback matrix, that is,  solve
\[
\inf_{K} \, {\|K\|} \quad \text{ such that } \quad A-BK \text{ is stable},
\]
where ${\|\cdot\|}$ is a given norm such as the $\ell_2$ norm or the Frobenius norm. 
\end{enumerate}

The second goal is to consider the analogous problem 
for system triplets $(A,B,C)$, referred to as the static-output feedback (SOF) problem. The SOF problem requires to find $K \in \R^{m,p}$ such that
$A-BKC$ is stable; see~\cite{SyrADG97} for a survey on the SOF problem. This decision problem is believed to be NP-hard as no polynomial-time algorithm is known. 
Moreover, if extra constraints are imposed on the entries of the static controller, then this decision problem is NP-hard~\cite{Nem93,BloT97}. As a consequence, the minimal-norm SOF problem, for which the norm of $K$ is minimized, is hard as well. 
For more discussion on the hardness of this problem, we refer to the recent paper by Peretz~\cite{Per16} and the references therein. 
The solution of the SOF problem is important for systems which models
structural dynamics, and naturally needs a static feedback that can be built into the structure~\cite{SpeS97,XuT02,YanLJ03,PolKS03,Per16}.
It was shown that optimal SOFs 
may achieve similar performance as optimal dynamic feedbacks. 

\subsection{Contribution and outline of the paper}

Recently, in~\cite{GilS17}, a parametrization of the set of all stable matrices was obtained in terms of dissipative Hamiltonian (DH) systems.
DH systems are special cases of port-Hamiltonian systems, which recently have received a lot attention in energy based modeling; see for example \cite{GolSBM03,Sch06,SchM13}, and also  \cite{GilS17b,BeaMV18,MehV19} for robustness analysis. 
A matrix
$A \in \R^{n,n}$ is called a \emph{DH matrix} if $A=(J-R)Q$ for some $J, R, Q \in \R^{n,n}$ such that
$J^T=-J$, $R$ is positive semidefinite and $Q$ is positive definite. A matrix $A$ is stable if and only if it is a DH matrix. This parametrization has been used to solve several nearness problems for LTI systems~\cite{MehMS16,GilMS17,GilS17b,GilKS18a}.  
In this paper, we provide a complete characterization
for the static-state stabilizing feedbacks of a given pair $(A,B)$ in terms of DH matrices. 
This provides an efficient
way to check the feasibility by solving a convex semidefinite program (SDP).
We then propose a sequential SDP (SSDP) method to solve the optimization problem of minimizing the norm of the feedbacks. 
We extend this approach to tackle the SOF problem for the system triplet $(A,B,C)$ and compare our solution with the algorithms proposed in~\cite{H2hifoo} and~\cite{Per16}. 
Although we cannot guarantee to obtain a feasible solution in all cases due to the complexity of the problem, we are able in many cases to get better solutions (that is, provide a feasible solution with smaller norm). 
 
The paper is organized as follows. 
 In Section~\ref{sec:prelim}, 
 we state some preliminary 
results from the literature. 
In Section~\ref{sec:pair}, we give a complete characterization 
of the static-state stabilizing feedbacks for a system pair $(A,B)$ in terms of 
DH matrices. 
In Section~\ref{sec:triplet}, 
we extend these results for the SOF problem of a system triplet $(A,B,C)$. 
In Section~\ref{sec:algorithm}, using our characterizations
based on DH matrices, 
we propose two algorithms 
to minimize the norm of stabilizing feedback matrices: 
one for the SSF problem (Algorithm~\ref{algo:stabpair}) and one for the SOF problem (Algorithm~\ref{algo:stabtriplet}). 
In Section~\ref{sec:numexp},  we illustrate the effectiveness of Algorithm~\ref{algo:stabtriplet} compared to the two state-of-the-art methods for the SOF problem proposed in~\cite{H2hifoo, Per16} on numerous numerical examples 
from the COMPLeIB library~\cite{leibfritz2004compleib}. 
Numerical results for  Algorithm~\ref{algo:stabpair} on the SSF problem are reported in Appendix~\ref{appA}.

\paragraph{Notation}

Throughout the paper, $X^T$ and $X^\dagger$  stand for the transpose and Moore-Penrose pseudoinverse of a real matrix $X$, respectively.
We write $X\succ 0$ and $X\succeq 0$ $(X \preceq 0)$ if $X$ is symmetric and positive definite
or positive semidefinite (symmetric negative semidefinite), respectively.
By $I_m$ we denote the identity matrix of size $m \times m$, and by $\text{null}(X)$ we denote the null space of $X$.
The eigenvalue spectrum of a matrix $X$ is denoted by $\Lambda(X)$.
For a given matrix triplet $(A,B,C)$, where $A\in \R^{n,n}$, $B\in \R^{n,m}$ and $C \in \R^{p,n}$, we define
\begin{equation}
\mathcal K(A,B,C):=\left\{K\in \R^{m,p}~|~A-BKC~\text{is~stable}\right\}.
\end{equation}
and
\begin{equation}
\mathcal W(A,B,C):=\left\{W\in \R^{n,n}~|~A-BB^{\dagger}WC^\dagger C~\text{is~stable}\right\}.
\end{equation}
For a matrix pair $(A,B)$, we use the notation $\mathcal K_R(A,B)$ for $\mathcal K(A,B,I_p)$ and $\mathcal W_R(A,B)$ for $\mathcal W(A,B,I_p)$. For the matrix triplet $(A,I_m,C)$, we use the notation $\mathcal K_L(A,C)$ for $\mathcal K(A,I_m,C)$ and $\mathcal W_L(A,C)$ for $\mathcal W(A,I_m,C)$. 


\section{Preliminaries} \label{sec:prelim}

In this section, we present results from the literature that will be useful in the following sections. 

Let us first recall the definition of a DH matrix: 
A matrix $A \in \mathbb R^{n,n}$ is said to be a DH matrix if $A=(J-R)Q$ for some $J,R,Q \in \mathbb R^{n,n}$ such that
$J^T=-J$, $R\succeq 0$ and $Q \succ 0$. The set of stable matrices is characterized as the set of DH matrices in the following.

\begin{theorem}{\rm \cite[Lemma 2]{GilS17}}\label{thm:gilsstab}
Let $A\in \R^{n,n}$. Then $A$ is stable if and only if $A$ is a DH matrix.
\end{theorem}

For DH matrices, we can easily derive the following lemma for which we provide the proof which will be useful later on.  
\begin{lemma}{\rm \cite[Lemma 3]{GilMS17}}  \label{lem:unitequi}
DH matrices are invariant under orthogonal transformations.
\end{lemma}
\begin{proof} 
 Let $A$ be a DH matrix, that is, 
$A=(J-R)Q$ for some $J^T=-J$, $R\succeq 0$, and $Q\succ 0$. Let $U$ be orthogonal, that is, $U^TU=I_n=UU^T$. Then
$U^TAU=U^T(J-R)QU=(U^TJU-U^TRU)U^TQU$ is a DH matrix since $(U^TJU)^T=-U^TJU$, $U^TRU \succeq 0$, and $U^TQU \succ 0$.
\end{proof}


The following lemma will be used to parametrize the set of all feedbacks that stabilize a system pair $(A,B)$ in Section~\ref{sec:pair} and system triplet $(A,B,C)$ in Section~\ref{sec:triplet}. 

\begin{lemma}{\rm{\cite[Lemma 1.3]{Sun93}}}\label{lem:Suns_result}
Let $A \in \C^{p , m}$, $B \in \C^{n , q}$, $C \in \C^{p , q}$, and
\[
 \Upsilon=\big\{E\in \C^{m , n}\,\big|\,AEB=C\big\}.
\]
Then $\Upsilon \neq \emptyset $ if and only if $A,B,C$ satisfy
$AA^{\dagger}CB^{\dagger}B=C$. If the latter condition is satisfied then
\[
\Upsilon=\big\{A^{\dagger}CB^{\dagger}+Z-A^{\dagger}AZBB^{\dagger}\,\big|\, Z\in\C^{n, n}\big\},
\]
and
\[
\min_{E\in \Upsilon}{\|E\|}_F={\|A^{\dagger}CB^{\dagger}\|}_F,\quad
\min_{E\in \Upsilon}{\|E\|}_2={\|A^{\dagger}CB^{\dagger}\|}_2.
\]
\end{lemma}

The following well-known lemma gives an equivalent characterization for a positive semidefinite
matrix; it will be used in establishing new conditions for the existence of stabilizing feedbacks. 

\begin{lemma}{\rm \cite{Alb69} }\label{lem:psd_character}
Let the integer $s$ be such that $0<s<n$, and $R=R^T \in \R^{n,n}$ be partitioned as
$R=\mat{cc}B & C^T\\C & D\rix$ with $B\in \R^{s,s}$, $C\in \R^{n-s,s}$ and $D \in \R^{n-s,n-s}$. Then $R \succeq 0$ if and only if
\[
i)~D \succeq 0, \quad ii)~\operatorname{null}(D) \subseteq \operatorname{null}(C^T),\quad {\rm and} \quad iii)~ A-C^TD^{\dagger}C \succeq 0.
\]
\end{lemma}

\section{DH characterization of static-state stabilizing feedbacks} \label{sec:pair}



Let us denote the set of triplets $(J,R,Q)$ that form a DH matrix as follows 
\begin{equation}
\mathbb D\mathbb H_{\succ}^n:=\left\{(J,R,Q)\in (\R^{n,n})^3~|~J^T=-J,~R\succeq 0,~Q\succ 0 \right\}. 
\end{equation}
For a triplet $(J,R,Q) \in (\R^{n,n})^3$, let us also define
\begin{equation}
g(J,R,Q):=B^{\dagger}(A-(J-R)Q).
\end{equation} 
Using Lemma~\ref{lem:Suns_result}, we have the following characterization of the set $\mathcal K_{R}(A,B)$
in terms of triplets $(J,R,Q) \in \mathbb D\mathbb H_\succ^n$.

\begin{theorem}\label{thm:completechar}
Let $A\in \R^{n,n}$ and $B\in \R^{n,m}$.  
Then, 
\begin{equation}\label{eq:equichar}
\mathcal K_R(A,B)=\left \{g(J,R,Q)-(I_m\text{$-$}B^{\dagger}B)Y 
\Big| (J,R,Q) \in \mathbb D\mathbb H_{\succ}^n,(I_n\text{$-$}BB^\dagger)(A\text{$-$}(J\text{$-$}R)Q)\text{$=$}0, Y\in \R^{m,n} \right \}.
\end{equation}
\end{theorem}
\begin{proof} 
Let us first show that $\mathcal K_R(A,B) \neq \emptyset$ if and only if there exists $(J,R,Q) \in \mathbb D\mathbb H_{\succ}^n$  such that $(I_n-BB^{\dagger})(A-(J-R)Q)=0$. 
Let $K\in \mathcal K_R(A,B)$, that is, $A-BK$ is stable. Then by Theorem~\ref{thm:gilsstab}, there exists $(J,R,Q)\in \mathbb D\mathbb H_{\succ}^n$
such that $A-BK=(J-R)Q$. This implies that $BK=A-(J-R)Q$ and thus
\[
BB^\dagger(A-(J-R)Q)=BB^\dagger BK=BK=A-(J-R)Q,
\]
since $BB^\dagger B=B$. Conversely, let $(I_n-BB^{\dagger})(A-(J-R)Q)=0$, and consider $K=B^\dagger (A-(J-R)Q)$. Then
$K$ satisfies
\[
A-BK=A-BB^\dagger(A-(J-R)Q)=A-A+(J-R)Q=(J-R)Q.
\]
This implies that $A-BK$ is a DH matrix and thus stable, hence $K \in \mathcal K_R(A,B)$. 

To show~\eqref{eq:equichar}, let $K$ be a matrix of the form $g(J,R,Q)-(I_m-B^{\dagger}B)Y$ satisfying $(I_n-BB^{\dagger})(A-(J-R)Q)=0$. We have that $A-BK=(J-R)Q$, and thus $K \in \mathcal K_R(A,B)$
because $(J,R,Q)\in \mathbb D\mathbb H_{\succ}^n$. This proves the inclusion $``\supseteq"$. 
For the inclusion $``\subseteq"$, let $K\in \mathcal K_R(A,B)$. By Theorem~\ref{thm:gilsstab} there exists $(J,R,Q)\in \mathbb D\mathbb H_{\succ}^n$
such that $BK=A-(J-R)Q$. In view of Lemma~\ref{lem:Suns_result} there exists $Y\in \R^{m,n}$
such that $K=g(J,R,Q)-(I_m-B^{\dagger}B)Y$ and $(I_n-BB^{\dagger})(A-(J-R)Q)=0$. This proves $``\subseteq"$. 
\end{proof}

\begin{corollary}\label{cor:1}
For a fixed triplet $(J,R,Q) \in \mathbb D\mathbb H_{\succ}^n $, define
\begin{equation}\label{eq:equichar2}
 \widetilde{\mathcal K}_R(J,R,Q):=\left \{K\in\mathcal K_R(A,B) \big|~K=g(J,R,Q)-(I_m-B^{\dagger}B)Y,\,Y\in \R^{m,n}\right \}.
\end{equation}
If $\widetilde{\mathcal K}_R(J,R,Q) \neq \emptyset$, then
\begin{equation}\label{eq:minnormjrq}
\min_{K\in  \widetilde{\mathcal K}_R(J,R,Q)} {\|K\|}_F={\|g(J,R,Q)\|}_F,\quad \text{and}\quad \min_{K\in \widetilde{\mathcal K}_R(J,R,Q)}{\|K\|}_2={\|g(J,R,Q)\|}_2.
\end{equation}
\end{corollary}
\begin{proof}
This follows immediately from 
Lemma~\ref{lem:Suns_result} and 
Theorem~\ref{thm:completechar}.
\end{proof}

The following theorem gives an alternative way to check the existence of a stabilizing feedback. 
It can be useful in numerical methods as it reduces the number of variables depending on the rank of $B$; see also  Remark~\ref{rem:smallsize}.

\begin{theorem}\label{thm:equitemp}
Let $A\in \R^{n,n}$ and $B\in \R^{n,m}$. Let $U$ be an orthogonal matrix such that $U^TBB^{\dagger}U=\mat{cc}I_k&0\\0&0\rix$,
where $\operatorname{rank}(B)=\operatorname{rank}(BB^{\dagger})=k$. Let $\hat A=U^TAU=\mat{cc}\hat A_{11}&\hat A_{12}\\\hat A_{21}&\hat A_{22}\rix$,
where $\hat A_{11} \in \R^{k,k}$ and $\hat A_{22}\in \R^{n-k,n-k}$.
Then $\mathcal K_R(A,B)\neq \emptyset$ if and only if there exist $\hat Q(\succ 0)\in \R^{n,n}$,
$\hat J_{22}, \hat R_{22} \in \R^{n-k,n-k}$ and $ \hat J_{21}, \hat R_{22} \in \R^{n-k,k} $ such that $\hat J_{22}^T=-\hat J_{22}$,
$\hat R_{22}\succeq 0$, $\operatorname{null}(\hat R_{22})\subseteq \operatorname{null}(\hat R_{21}^T)$, and
\begin{equation}\label{eq:concl}
\mat{cc}\hat A_{21}& \hat A_{22}\rix= \mat{cc}\hat J_{21}-\hat R_{21}& \hat J_{22}-\hat R_{22}\rix \hat Q.
\end{equation}
\end{theorem}
\begin{proof} First suppose that $\mathcal K_R(A,B) \neq \emptyset$, and let $K \in \mathcal K_R(A,B)$.
Then from Theorem~\ref{thm:completechar} there exists $(J,R,Q)\in \mathbb D\mathbb H_\succ^n$ such that
\begin{equation}\label{eq:diff1}
(I_n-BB^\dagger)(A-(J-R)Q)=0.
\end{equation}
Multiplying~\eqref{eq:diff1} by $U^T$ from the left and by $U$ from the right, by using the fact that $U^TU=I_n$,
and by setting
\[
\hat J=U^TJU=\mat{cc}\hat J_{11} & -\hat J_{21}^T\\ \hat J_{21}& \hat J_{22}\rix,
\hat R=U^TRU=\mat{cc}\hat R_{11} & \hat R_{21}^T\\ \hat R_{21}& \hat R_{22}\rix,
\hat Q=U^TQU=\mat{cc}\hat Q_{11} & \hat Q_{21}^T\\ \hat Q_{21}& \hat Q_{22}\rix,
\]
we have that
\[
\mat{cc}0&0 \\0 & I_{n-k}\rix
\Big(
\mat{cc}\hat A_{11}&\hat A_{12}\\\hat A_{21}&\hat A_{22}\rix-
\mat{cc}\hat J_{11}-\hat R_{11} & -\hat J_{21}^T-\hat R_{21}^T\\ \hat J_{21}-\hat R_{21}& \hat J_{22} - \hat \R_{22}\rix
\mat{cc}\hat Q_{11} & \hat Q_{21}^T\\ \hat Q_{21}& \hat Q_{22}\rix
\Big)=0.
\]
This implies that
\[
\mat{cc}\hat A_{21}& \hat A_{22}\rix= \mat{cc}\hat J_{21}-\hat R_{21}& \hat J_{22}-\hat R_{22}\rix \hat Q.
\]
Note that $\hat Q \succ 0$ and $\hat J_{22}^T=-\hat J_{22}$, because $U$ is orthogonal and $(J,R,Q)\in \mathbb D\mathbb H_\succ^n$.
Also from Lemma~\ref{lem:psd_character} $\hat R_{22}\succeq 0$ and
$\text{null}(\hat R_{22})\subseteq \text{null}(\hat R_{21}^T)$, since $\hat R\succeq 0$. This proves the ``if"
part.

Conversely, suppose that~\eqref{eq:concl} holds. Let $\hat R_{11}, \hat J_{11} \in \R^{k,k}$ be chosen such that
$\hat J_{11}^T=-\hat J_{11}$ and $\hat R_{11}-\hat R_{21}\hat R_{22}^{\dagger}\hat R_{21} \succeq 0$
(for example: $\hat R_{11}=\hat R_{21}\hat R_{22}^{\dagger}\hat R_{21}$) and let
$\hat J=U^TJU=\mat{cc}\hat J_{11} & -\hat J_{21}^T\\ \hat J_{21}& \hat J_{22}\rix$ and
$\hat R=U^TRU=\mat{cc}\hat R_{11} & \hat R_{21}^T\\ \hat R_{21}& \hat R_{22}\rix$.
Observe that $\hat J^T=-\hat J$, $\hat R \succeq 0$ (Lemma~\ref{lem:psd_character}) and
\begin{equation}\label{eq:temp1}
\mat{cc}0&0\\0&I_{n-k}\rix(\hat A-(\hat J-\hat R)\hat Q)=0.
\end{equation}
Multiplying~\eqref{eq:temp1} by $U$ from the left and by $U^T$ from the right, and by using the fact that $U^TU=I_n$,
we get
\[
(I_{n}-BB^{\dagger})(A-(U\hat J U^T- U\hat R U^T) U\hat Q U^T)=0.
\]
Thus from Theorem~\ref{thm:completechar}, $K=B^{\dagger}(A-(U\hat J U^T- U\hat R U^T) U\hat Q U^T)$
satisfies $A-BK=(U\hat J U^T- U\hat R U^T) U\hat Q U^T$. This implies from Lemma~\ref{lem:unitequi} that $A-BK$
is a DH matrix and thus $K \in \mathcal K_{R}(A,B)$.
\end{proof}

It is well known that if the system pair $(A,B)$ is controllable, then $\mathcal K_R(A,B) \neq \emptyset$; see~\cite{TreSAH12}.
In the following, we obtain a different sufficient condition for the existence of a stabilizing feedback for pair $(A,B)$ which is a corollary of Theorem~\ref{thm:equitemp}.
\begin{corollary}
Let $A\in \R^{n,n}$ and $B\in \R^{n,m}$. Let $U$ be an orthogonal matrix such that $U^TBB^{\dagger}U=\mat{cc}I_k&0\\0&0\rix$,
where $\operatorname{rank}(B)=\operatorname{rank}(BB^{\dagger})=k$. Let $\hat A=U^TAU=\mat{cc}\hat A_{11}&\hat A_{12}\\\hat A_{21}&\hat A_{22}\rix$,
where $\hat A_{11} \in \R^{k,k}$ and $\hat A_{22}\in \R^{n-k,n-k}$. If $\hat A_{22}$ is stable, then
$\mathcal K_R(A,B)\neq \emptyset$.
\end{corollary}
\begin{proof}  Since $\hat A_{22}$ is stable, from Theorem~\ref{thm:gilsstab} there exist
$(\hat J_{22},\hat R_{22},\hat Q_{22}) \in \mathbb D \mathbb H_\succ^{n-k}$ such that $A_{22}=(\hat J_{22}-\hat R_{22})\hat Q_{22}$.
Define
\[
\hat J=\mat{cc}\hat J_{11} & -\hat J_{21}^T\\ \hat J_{21}& \hat J_{22}\rix,
\hat R=\mat{cc}0 & 0\\ 0& \hat R_{22}\rix,
\hat Q=\mat{cc}I_k & 0\\ 0& \hat Q_{22}\rix,
\]
where $\hat J_{11} \in \R^{k,k}$ is any skew-symmetric matrix and $\hat J_{21}=\hat A_{21}$.
Note that $(\hat J,\hat R,\hat Q)\in \mathbb D \mathbb H_\succ^n$.
Further define
\[
J=U\hat J U^T,\quad R=U\hat R U^T,\quad \text{and}\quad Q=U\hat Q U^T.
\]
Proceeding as in Theorem~\ref{thm:equitemp}, we have that
$(I_n-BB^\dagger)(A-(J-R)Q)=0$, and the matrix $K =B^\dagger(A-(J-R)Q) $ stabilizes the system pair $(A,B)$ hence $K\in \mathcal K_R(A,B)$.
\end{proof} 

Note that the converse of the above corollary does not hold.  For example, consider 
\[
A =\mat{cccc}
   -0.3633 &  -0.2867 &  -0.7294 &  -2.2033\\
   -1.0206 &  -0.1973 &   1.1473 &  -0.5712\\
   -3.0730 &   0.4056 &   0.5979 &   0.2140\\
    0.6263 &  -1.4193 &  -1.2813 &   0.9424
    \rix \text{ and }  
B =\mat{cc}
    0.0937 &  -0.9610\\
   -1.1223 &  -0.6537\\
    0.3062 &  -1.2294\\
   -1.1723 &  -0.2710
   \rix, 
\]
where $\Lambda(A)=   \{1.9604 + 1.8099i,1.9604 + 1.8099i, -1.8378,-1.1032 \}$. We have 
\[
B^{\dagger}=    \mat{cc} 0.1089 & -0.3490\\ -0.3787 & -0.1472\\   0.2103 & -0.4607\\ -0.4269& -0.0073
   \rix^T~\text{and}\quad
    U = \mat{cccc}
   -0.5879  &  0.0000 &   0.7931 &   0.1595\\
   -0.1803  & -0.6991 &   0.0055 &  -0.6919\\
   -0.7866  &  0.1098 &  -0.6010 &   0.0893\\
    0.0562  & -0.7066 &  -0.0988 &   0.6985
    \rix.
\]
For the pair $(A,B)$, the SSF problem is feasible, that is, $\mathcal K_R(A,B)\neq \emptyset$ as
\[
   K =\mat{cccc}
    3.4237 &  27.5800 &   1.9374 & -35.6683\\
   10.1752 & -23.1448 & -11.0932 &  20.1984
   \rix
\]
with $\Lambda(A-BK)=   \{-9.9703,-8.7173,-2.4493,-3.1767\}$. However, for the
transformed matrix
\[
\hat A=U^TAU=\mat{cc} \hat A_{11} & \hat A_{12}\\ \hat A_{21} & \hat A_{22}\rix, \quad \text{where}\quad
\hat A_{22}=\mat{cc}
    1.8596 &  -0.9818\\
    1.6202 &   1.0747
    \rix
\]
we have $\Lambda(A_{22})= \{1.4672 + 1.1986i, 1.4672 - 1.1986i\}$. This implies
$\mathcal K_R(A,B) \neq \emptyset$ while $\hat A_{22}$ is not stable.

In the following theorem, 
we summarize the various necessary and sufficient conditions for the existence of a static stabilizing feedback for a given system pair $(A,B)$.

\newpage 

\begin{theorem}
Let $A\in \R^{n,n}$ and $B\in \R^{n,m}$. Then the following are equivalent.
\begin{enumerate}
\item $\mathcal W_R(A,B) \neq \emptyset$.
\item $\mathcal K_R(A,B) \neq \emptyset$.
\item There exists $(J,R,Q) \in \mathbb D\mathbb H_{\succ}^n$  such that $A-BK=(J-R)Q$
for some $K\in \R^{m,n}$.
\item There exists $(J,R,Q) \in \mathbb D\mathbb H_{\succ}^n$ such that $A-BB^{\dagger}W=(J-R)Q$
for some $W\in \R^{n,n}$.
\item There exists $(J,R,Q) \in \mathbb D\mathbb H_{\succ}^n$  such that $(I_n-BB^{\dagger})(A-(J-R)Q)=0$.
\item There exist $\hat Q(\succ 0)\in \R^{n,n}$,
$\hat J_{22}, \hat R_{22} \in \R^{n-k,n-k}$ and $ \hat J_{21}, \hat R_{22} \in \R^{n-k,k} $ such that $\hat J_{22}^T=-\hat J_{22}$, $\hat R_{22}\succeq 0$, $\operatorname{null}(\hat R_{22})\subseteq \operatorname{null}(\hat R_{21}^T)$, and
\begin{equation*}
\mat{cc}\hat A_{21}& \hat A_{22}\rix= \mat{cc}\hat J_{21}-\hat R_{21}& \hat J_{22}-\hat R_{22}\rix \hat Q,
\end{equation*}
where $\hat A=U^TAU=\mat{cc}\hat A_{11}&\hat A_{12}\\\hat A_{21}&\hat A_{22}\rix$ with $\hat A_{11} \in \R^{k,k}$, $\hat A_{22}\in \R^{n-k,n-k}$ and 
$U$ is an orthogonal matrix such that $U^TBB^{\dagger}U=\mat{cc}I_k&0\\0&0\rix$ with $\operatorname{rank}(B)=k$.
\end{enumerate}
\end{theorem}
\begin{proof} 
$1)\Leftrightarrow 2)$: This follows immediately from the relationship between the two sets $\mathcal W_R(A,B)$ and $\mathcal K_R(A,B)$. In fact, 
for any $W\in \mathcal W_R(A,B)$ and $Y \in \R^{m,n}$, we have 
$K=B^{\dagger}W+(I_m-B^{\dagger}B)Y \in K_R(A,B)$, while for any $K\in \mathcal K_R(A,B)$
and $N\in \R^{n,n}$, we have 
 $W=BK+(I_n-BB^{\dagger})N \in \mathcal W_R(A,B)$~\cite{Per12}. 

$3)\Rightarrow 2)$ and $4)\Rightarrow 1)$ follow trivially, and $2)\Rightarrow 3)$ and $1)\Rightarrow 4)$
follow from Theorem~\ref{thm:gilsstab}.

$2)\Leftrightarrow 5)$: This follows from Theorem~\ref{thm:completechar}.

$2)\Leftrightarrow 6)$: This follows from Theorem~\ref{thm:equitemp}.
\end{proof} 

\section{DH characterization of static-output stabilizing feedbacks}\label{sec:triplet}

We first state a result similar to Theorem~\ref{thm:completechar} that characterizes
the set of SOFs in terms of DH matrices.  For this, let $(J,R,Q) \in (\R^{n,n})^3$ and define
\begin{equation}
f(J,R,Q):=B^{\dagger}(A-(J-R)Q)C^{\dagger}.
\end{equation}
\begin{theorem}\label{thm:completecharABC}
Let $(A,B,C)$ be a system triplet. Then 
\[
K(A,B,C) = \left\{
f(J,R,Q)\text{$+$}Z\text{$-$}B^{\dagger}BZCC^\dagger 
\Big|  
(J,R,Q) \in \mathbb D\mathbb H_{\succ}^n,
B f(J,R,Q) C\text{$=$}A\text{$-$}(J\text{$-$}R)Q, 
Z \in \R^{m \times p} 
\right\}. 
\]
\end{theorem}
\begin{proof} 
 The proof is similar to Theorem~\ref{thm:completechar}.
\end{proof}

A corollary similar to Corollary~\ref{cor:1} is given in the following.
\begin{corollary}
For a fixed triplet $(J,R,Q) \in \mathbb D\mathbb H_{\succ}^n $, define
\begin{equation}\label{eq:equichar2ABC}
 {\widetilde{ \mathcal K}}(J,R,Q):=\left \{K\in\mathcal K(A,B,C) \big|~K=f(J,R,Q)+Z-B^{\dagger}BZCC^\dagger,\,Z\in \R^{m,p}\right \}.
\end{equation}
If $\widetilde{\mathcal K}(J,R,Q) \neq \emptyset$, then
\begin{equation}\label{eq:minnormjrqABC}
\min_{K\in  \widetilde{\mathcal K}(J,R,Q)} {\|K\|}_F={\|f(J,R,Q)\|}_F,\quad \text{and}\quad \min_{K\in \widetilde{\mathcal K}(J,R,Q)}{\|K\|}_2={\|f(J,R,Q)\|}_2.
\end{equation}
\end{corollary}

It is easy to see that if the matrix triplet $(A,B,C)$ is stabilizable, then the pairs
$(A,B)$ and $(A,C)$ are necessarily stabilizable. Indeed, if $K \in \mathcal K(A,B,C)$, then
$KC \in \mathcal K_R(A,B)$ and $BK \in \mathcal K_L(A,C)$. In general, however, stabilization of
$(A,B)$ and $(A,C)$ do not guarantee the stabilization of $(A,B,C)$. To ensure this, there must exist $Y \in \mathcal W_R(A,B)$ and a $Z \in \mathcal W_L(A,C)$
 such that $BB^\dagger Y=ZC^\dagger C$~\cite[Lemma 3.1]{Per16}. 
In the following, we give a different sufficient condition for the stabilizability of $(A,B,C)$  in terms of DH matrices; if 
 the stabilizability of $(A,B)$ and $(A,C)$ is determined simultaneously by the same stable matrix, that is, 
if there exists a $Y \in \mathcal K_R(A,B)$ and a $Z \in \mathcal K_L(A,C)$ such that
$A-BY=A-ZC$, then $(A,B,C)$ is stabilizable.
\begin{theorem}\label{thm:simulcon}
Let $(A,B,C)$ be a given system triplet. 
Then $(A,B,C)$ is stabilizable
if and only if there exists a $(J,R,Q) \in \mathbb D\mathbb H_{\succ}^n$ such that
\begin{equation}\label{eq:simul}
(I_n-BB^\dagger)(A-(J-R)Q)=0\quad \text{and}\quad (A-(J-R)Q)(C^\dagger C-I_n)=0.
\end{equation}
 If the later conditions are 
satisfied, then for all SOFs related to  such $(J,R,Q)$ we have
\begin{equation}\label{jointcond}
\widetilde{\mathcal K}(J,R,Q)=\left \{
f(J,R,Q)  +(I_m-B^\dagger B)Y + Z(I_s-CC^\dagger)~\big|~
Y,Z\in \R^{m,s} \right\}.
\end{equation}
\end{theorem}
\begin{proof}
It is easy to check that for a given $(J,R,Q) \in \mathbb D\mathbb H_{\succ}^n$ satisfying~\eqref{eq:simul}, the matrix $f(J,R,Q) \in  \mathcal K(A,B,C)$.
Conversely, let $K \in \mathcal K(A,B,C)$, that is, $A-BKC$ is stable. This implies 
$A-BKC=(J-R)Q$ for some $(J,R,Q)\in \mathbb D\mathbb H_{\succ}^n$, and thus
$KC \in \mathcal K_R(A,B)$ and $BK \in \mathcal K_L(A,C)$.  Hence
as an application of Lemma~\ref{lem:Suns_result}, $(J,R,Q)$ satisfies~\eqref{eq:simul}. 
Further,~\eqref{jointcond} follows immediately by using the fact that 
$BB^\dagger B=B$ and $CC^\dagger C=C$.
\end{proof}

\section{Computing stabilizing feedback matrices}\label{sec:algorithm}

In this section, we exploit the results obtained in the previous sections and present a new framework
based on DH matrices to attack the SSF and SOF problems.

\subsection{Stabilization of a system pair $(A,B)$}

In this section, we focus on stabilizing a system pair $(A,B)$, that is, finding $K$ such that $A-BK$ is stable. 
We propose Algorithm~\ref{algo:stabpair} which consists in two mains steps described in the next subsections: first finding a feasible solution and then improving this solution.

\subsubsection{Feasibility problem}

A necessary and sufficient condition was obtained in Theorem~\ref{thm:completechar} 
for the feasibility of the static feedback problem in terms of DH matrices: 
$\mathcal K_R(A,B) \neq \emptyset$ if and only if there exists
$(J,R,Q)\in \mathbb D\mathbb H_\succ^n$ such that $(I_n-BB^\dagger)(A-(J-R)Q)=0$. 
Trying to find a feasible solution of the latter equation  can be done by considering the following optimization problem 
\begin{equation}\label{eq:algo1ab}
\mu:=\inf_{J,R,Q\in \R^{n,n}, J^T=-J,R\succeq 0,Q\succ 0}{\left \|(I_n-BB^\dagger)(A-(J-R)Q)\right \|}, 
\end{equation}
and checking whether $\mu = 0$, that is, $\mathcal K_R(A,B) \neq \emptyset$ if and only if $\mu =0$. Since $Q$ is positive definite, we have that 
\[
(I_n-BB^\dagger)(A-(J-R)Q)=0 
\quad \iff \quad 
(I_n-BB^\dagger)(AQ^{-1}-J+R)=0. 
\] 
Let $P=Q^{-1}$, and define 
\begin{equation}\label{eq:algo2ab}
\tilde{\mu} = \inf_{J,R,Q\in \R^{n,n}, J^T=-J,R\succeq 0,P\succ 0}{\left \|(I_n-BB^\dagger)(AP-J+R)\right \|}.
\end{equation}
We have that $\mu = 0$ if and only if $\tilde{\mu} = 0$. 
Note that there is a scaling degree of freedom between $(J,R)$ and $P$ since $(\alpha J, \alpha R)$ and $P/\alpha$ for any $\alpha > 0$ leads to an equivalent solution. Hence, we may assume $P \succ I_n$. 
Note also that the feasible set in~\eqref{eq:algo2ab} is neither open (due to constraint $R\succeq 0$) nor closed (due to
constraint $P\succ 0$) and replacing the feasible set in~\eqref{eq:algo2ab} by its closure does
not change the value of the infimum in~\eqref{eq:algo2ab}, hence 
\begin{equation}\label{eq:algo3ab}
\tilde{\mu} = \inf_{J,R,Q\in \R^{n,n}, J^T=-J,R\succeq 0,P\succeq I_n}{\left \|(I_n-BB^\dagger)(AP-J+R)\right \|}.
\end{equation}
Thus checking feasibility is equivalent to check that the value of $\tilde{\mu}$ in~\eqref{eq:algo3ab}
is zero or not. 
The problem~\eqref{eq:algo3ab} is convex. More precisely, it is a semidefinite program (SDP) which can be solved efficiently with dedicated solvers. 

\begin{remark}{\rm
Recall the classical result that for any controllable pair $(A,B)$, there exists $K$
such that $A-BK$ is stable~\cite{TreSAH12}. 
However, if $(A,B)$ is not controllable, one cannot conclude that $\mathcal{K}_R(A,B) = \emptyset$.   
Note that, as opposed to previous works in the literature~\cite{Per12}, no assumption was made on the pair $(A,B)$ while obtaining our reformulation~\eqref{eq:algo3ab}. 
Therefore~\eqref{eq:algo3ab} provides a general way to check the existence of static feedback and works even when $(A,B)$ is not controllable. 
}
\end{remark}

\begin{remark}\label{rem:smallsize}{\rm
Note that one can use Theorem~\ref{thm:equitemp} to obtain an equivalent optimization problem with less variables (depending on the rank of $B$). Defining 
\begin{eqnarray*} 
&\nu:=\inf_{ J_{22}, R_{22} \in \R^{n-k,n-k}, J_{21} \in \R^{n-k,k}, Q\in \R^{n,n}}{\left\|\mat{cc}\hat A_{21}& \hat A_{22}\rix- \mat{cc} J_{21}-R_{21}&  J_{22}- R_{22}\rix Q\right\|}\\
&\text{such~ that}\quad Q\succ 0,~ J_{22}^T=- J_{22}, ~ R_{22}\succeq 0,~
\text{null}( R_{22})\subseteq \text{null}( R_{21}^T),  \nonumber
\end{eqnarray*}
we have $\mathcal K_R(A,B) \neq \emptyset$ if and only if $\nu=0$. Note that if $\hat R_{22}\succ 0$, then the condition $\text{null}(\hat R_{22})\subseteq \text{null}(\hat R_{21}^T)$
is always met. Therefore dropping this condition does not make any difference in our algorithm as the set of positive definite matrices is dense in the set of of positive semidefinite matrices. 
Therefore an equivalent reformulation of~\eqref{eq:algo3ab} is given by 
\begin{eqnarray*} 
&\nu=\inf_{ J_{22}, R_{22} \in \R^{n-k,n-k}, J_{21} \in \R^{n-k,k}, P\in \R^{n,n}}{\left\|\mat{cc}\hat A_{21}& \hat A_{22}\rix P- \mat{cc} J_{21}-R_{21}&  J_{22}- R_{22}\rix \right\|}\\
&\text{such~ that}\quad P\succeq I_n,~ J_{22}^T=- J_{22}, ~ R_{22}\succeq 0. \nonumber
\end{eqnarray*}
%
}
\end{remark}

\subsubsection{Optimization problem}

Suppose that the SSF problem is feasible, that is, $\mathcal K_R(A,B)\neq \emptyset$, and we want to solve the
minimization problem
\begin{equation}\label{eq:min1}
\inf_{K \in \mathcal K_R(A,B)} {\|K\|}. 
\end{equation}
In view of Theorem~\ref{thm:completechar} and Corollary~\ref{cor:1}, we can equivalently write~\eqref{eq:min1} as
\begin{eqnarray}  \label{eq:formulationP}
\inf_{(J,R,Q)\in \mathbb D \mathbb H_\succ^n}\, \inf_{K \in \widetilde{\mathcal K}_R(J,R,Q)} {\|K\|} 
= \inf_{(J,R,Q)\in \mathbb D \mathbb H_\succ^n,\, (I_n-BB^\dagger)(A-(J-R)Q)=0} {\|B^\dagger(A-(J-R)Q)\|}.  
\end{eqnarray}

In view of the formulation~\eqref{eq:formulationP}, a simple algorithm that can be used is a block coordinate descent (BCD) method: optimize alternatively over variables $(J,R)$ for $Q$ fixed, and $Q$ for $(J,R)$ fixed. In fact, the subproblems are SDPs hence can be solved efficiently. 
However, we have observed in practice that BCD seems to get stuck at saddle points; see Appendix~\ref{appA} for numerical results.   
For this reason, we have developed another algorithm that is based on sequential semidefinite programming (SSDP). 
As before, let us denote $P = Q^{-1}$, and reformulate~\eqref{eq:formulationP} as
\begin{equation}\label{eq:algo1}
\inf_{J,R,P\in \R^{n,n}, J^T=-J,R\succeq 0,P\succ 0, (I_n-BB^\dagger)(AP-(J-R))=0}
{\left \| B^\dagger ( A  - (J-R)  P^{-1}) \right \|}.
\end{equation}
Note that the feasible set is convex, with linear matrix inequalities (LMIs) and linear constraints. 
Given an initial solution $(J,R,P)$, 
we look for $(\Delta J,\Delta R, \Delta P)$ such that
$(J+\Delta J,R+\Delta R, P+\Delta P)$ is a better solution than $(J,R,P)$. 
To do so, we linearize the term $( A - (J+\Delta J-(R+\Delta R))  (P+\Delta P)^{-1})$ by using 
\[
(P+\Delta P)^{-1} \approx P^{-1} - P^{-1} \Delta P P^{-1}, 
\]
and removing the non-linear terms appearing in the product of the two components, that is, we use the following approximation: 
\begin{eqnarray*}
 A - (J+\Delta J-(R+\Delta R))  (P+\Delta P)^{-1}
\approx
 A - (J+\Delta J-(R+\Delta R)) P^{-1} +(J-R)  P^{-1} \Delta P P^{-1}.
\end{eqnarray*}
This results in the following optimization problem
\begin{align}
\inf_{\Delta J, \Delta R, \Delta P\in \R^{n,n}} &
{\left \|
B^\dagger \left(
A - (J+\Delta J)P^{-1}+(R+\Delta R) P^{-1} +(J-R)  P^{-1} \Delta P P^{-1}
\right)
\right \|}  \nonumber  \\
  \quad  \text{ such that }
& \quad
\Delta J^T=- \Delta J, R+\Delta R\succeq 0,P+\Delta P\succ 0, \label{eq:algo2}  \\
&  \quad
( I - B B^{\dagger} ) ( A \Delta P - (\Delta J - \Delta R) ) = 0 ,
\nonumber \\
&  \quad
\|\Delta J\| \leq \epsilon \| J\|,
\|\Delta R\| \leq \epsilon \| R\|,
\|\Delta P\| \leq \epsilon \| P\|. \nonumber
\end{align} 
Similar to a trust-region method, the value of $\epsilon$ is updated in the curse of the algorithm. As long as
the error of $(J+\Delta J,R+\Delta R,P+\Delta P)$ is larger than that of $(J,R,P)$, $\epsilon$ is decreased. For the next step, $\epsilon$ is increased to allow a larger trust-region radius.

\algsetup{indent=2em}
\begin{algorithm}[ht!]
\caption{Stabilizing a system pair $(A,B)$} \label{algo:stabpair}
\begin{algorithmic}[1]
\REQUIRE The $n$-by-$n$ matrix $A$, 
the $n$-by-$m$ matrix $B$, 
choice of a norm,  
step-size bound $0 < \underline{\epsilon} \ll 1$. 

\ENSURE An approximate solution $K$ 
to $\min_{K} {\|K\|}$ such that $A-BK$ is stable. \vspace{0cm}

\STATE \emph{\% Initialization phase} 

\STATE Initialize $\epsilon = 1$. 

\STATE Initialize $(J,R,P)$ as the optimal solution to~\eqref{eq:algo3ab}.

\STATE \emph{\% Optimization phase} 

\STATE Define $F(J,R,P) = {\left \| B^\dagger ( A  - (J-R)  P^{-1}) \right \|}$. 

\WHILE{some stopping criterion is met, or a maximum number of iterations is reached} 

\STATE Solve~\eqref{eq:algo2} to obtain $(\Delta J, \Delta R, \Delta P)$. 

\WHILE { 
$F(J+\Delta J,R+\Delta R,P+\Delta P) \geq F(J,R,P)$ 
and $\epsilon > \underline{\epsilon}$ 
}

\STATE Reduce $\epsilon$.

\STATE Solve~\eqref{eq:algo2} to obtain $(\Delta J, \Delta R, \Delta P)$.

\ENDWHILE

\STATE Set $(J,R,P) = (J+\Delta J,R+\Delta R,P+\Delta P)$

\STATE Increase $\epsilon$.

\ENDWHILE

\STATE Return $K = B^\dagger ( A  - (J-R)  P^{-1})$. 

\end{algorithmic} 
\end{algorithm}

\subsection{ Stabilization of a system triplet $(A,B,C)$ }

In this section, we focus on stabilizing a system triplet $(A,B,C)$, that is, finding $K$ such that $A-BKC$ is stable. We propose Algorithm~\ref{algo:stabtriplet} which, as for  Algorithm~\ref{algo:stabpair}, consists in two mains steps described in the next subsections.

\subsubsection{Feasibility problem}

In view of Theorem~\ref{thm:simulcon}, $\mathcal K(A,B,C) \neq \emptyset$ if and only if 
there exists $(J,R,Q)\in \mathbb D\mathbb H_\succ^n$ such that 
\[
(I_n-BB^\dagger)(A-(J-R)Q)=0\quad \text{and}\quad (A-(J-R)Q)(C^\dagger C-I_n)=0.
\] 
Finding a feasible solution can be done by considering the following optimization problem 
\begin{equation}\label{eq:refabc1}
\rho:=\inf_{J,R,Q\in \R^{n,n}, J^T=-J,R\succeq 0,Q\succ 0}{\left \|(I_n-BB^\dagger)(A-(J-R)Q)\right \|}
+ {\left \|(A-(J-R)Q)(C^\dagger C-I_n)\right \|}, 
\end{equation}
since $\mathcal K(A,B,C) \neq \emptyset$ if and only if $\rho=0$.  
Unlike~\eqref{eq:algo3ab},
a change of variable as $P=Q^{-1}$ in~\eqref{eq:refabc1} will not result in a convex optimization problem
due to the second term in~\eqref{eq:refabc1}. 
This was expected since the stabilization of a matrix triplet is hard; see Section~\ref{sec:intero}. 
Therefore to solve~\eqref{eq:refabc1},
we apply SSDP, as done for~\eqref{eq:algo1}, by solving, at each iteration, the following optimization problem: 
\begin{align} 
\inf_{\Delta J, \Delta R, \Delta P\in \R^{n,n}} &
{\left \|
(I_n-B B^\dagger) \left(
A - (J+\Delta J)P^{-1}+(R+\Delta R) P^{-1} +(J-R)  P^{-1} \Delta P P^{-1}
\right)
\right \|}  \nonumber  \\
& \quad + {\left \|
 \left(
A - (J+\Delta J)P^{-1}+(R+\Delta R) P^{-1} +(J-R)  P^{-1} \Delta P P^{-1}
\right)(C^\dagger C- I_n)
\right \|} \nonumber \\
  \quad  \text{ such that }
& \quad
\Delta J^T=- \Delta J, R+\Delta R\succeq 0,P+\Delta P\succ 0, \label{eq:tripletinit} \\
&  \quad
\|\Delta J\| \leq \epsilon \|\Delta J\|,
\|\Delta R\| \leq \epsilon \|\Delta R\|,
\|\Delta P\| \leq \epsilon \|\Delta P\|. \nonumber
\end{align}
Algorithm~\ref{algo:stabtriplet} (steps 1-16) summarizes this approach.

\subsubsection{Optimization problem}

In view of Theorem~\ref{thm:completecharABC}, we want to minimize the norm of feasible feedback $K$ by solving 
\begin{align}  
\inf_{(J,R,Q)\in \mathbb D \mathbb H_\succ^n} & \quad {\left\|B^\dagger(A-(J-R)Q)C^\dagger\right\|}, \nonumber \\ 
 \text{such that} & \quad  
 (I_n-BB^\dagger)(A-(J-R)Q)=0,  \text{ and }   \label{eq:formulationPabc} \\
 & \quad (A-(J-R)Q)(C^\dagger C-I_n)=0. \nonumber
\end{align}
To solve~\eqref{eq:formulationPabc}, we cannot use SSDP because the linear constraints cannot be linearized exactly (we would obtain an infeasible solution after one step). 
Therefore, in this case, we resort to BCD: alternatively solve \eqref{eq:formulationPabc} for $(J,R)$ with $Q$ fixed, and then for $Q$ with $(J,R)$ fixed; 
see steps 18-22 of Algorithm~\ref{algo:stabtriplet}.

\algsetup{indent=2em}
\begin{algorithm}[ht!]
\caption{Stabilizing a system triplet $(A,B,C)$} \label{algo:stabtriplet}
\begin{algorithmic}[1]
\REQUIRE The $n$-by-$n$ matrix $A$, 
the $n$-by-$m$ matrix $B$, 
the $p$-by-$n$ matrix $C$,  
choice of a norm, 
accuracy $0 < \delta \ll 1$, 
step-size bound $0 < \underline{\epsilon} \ll 1$, an initialization strategy (identity, random, ABI or ACI). 

\ENSURE If it succeeds, an approximate solution $K$ 
to $\min_{K} {\|K\|}$ such that $A-BKC$ is stable. \vspace{0.2cm}

\STATE \emph{\% Initialization phase} 

\STATE Initialize $(J,R,P)$ using the initialization strategy; see Section~\ref{init}. 

\STATE Initialize $\epsilon = 1$. 

\STATE Define $G(J,R,P) = 
\|(I_n-BB^\dagger)(A-(J-R)Q) \|
+ \|(A-(J-R)Q)(C^\dagger C-I_n) \|$. 

\WHILE{$G(J+\Delta J,R+\Delta R,P+\Delta P) \geq \delta$, or a maximum number of iterations is reached} 

\STATE Solve~\eqref{eq:tripletinit} to obtain $(\Delta J, \Delta R, \Delta P)$. 

\WHILE { 
$G(J+\Delta J,R+\Delta R,P+\Delta P) \geq G(J,R,P)$ 
and $\epsilon > \underline{\epsilon}$ 
}

\STATE Reduce $\epsilon$.

\STATE Solve~\eqref{eq:tripletinit} to obtain $(\Delta J, \Delta R, \Delta P)$.

\ENDWHILE

\STATE Set $(J,R,P) = (J+\Delta J,R+\Delta R,P+\Delta P)$

\STATE Increase $\epsilon$.

\ENDWHILE

\IF{$G(J,R,P) \geq \delta$}

\STATE The algorithm failed to find a feasible solution, STOP.  

\ENDIF

\STATE \emph{\% Optimization phase} 

\STATE $Q = P^{-1}$. 

\WHILE{some stopping criterion is met, or a maximum number of iterations is reached} 

\STATE Set $(J,R)$ as the optimal solution to~\eqref{eq:formulationPabc} for $Q$ fixed.  

\STATE Set $Q$ as the optimal solution to~\eqref{eq:formulationPabc} for $(J,R)$ fixed.   

\ENDWHILE

\STATE Return 
$K = B^\dagger ( A  - (J-R)  Q) C^\dagger$. 

\end{algorithmic} 
\end{algorithm}

\subsubsection{Initialization} \label{init}

Algorithm~\ref{algo:stabtriplet} needs to be initialized with some matrices $(J,R,P)$. 
We propose  four different initializations. In all cases, we choose $P$ and then set $(J,R)$ as an optimal solution of~\eqref{eq:refabc1} with $Q=P^{-1}$.  
\begin{enumerate}

\item Identity. We simply pick $P=I_n$. 

\item Random. We first generate the $n$-by-$n$ matrix $R$ where each entry is drawn using the normal distribution of mean 0 and standard deviation 1 (\texttt{randn(n,n)} in Matlab) and then set $P = (RR^T)^{1/2}$. (We use the square root so that $P$ has a smaller condition number.) 

\item ABI. We take $P$ as the optimal solution of~\eqref{eq:algo3ab}. This means that we choose $P$ such that there exists $J$ and $R$ where $K = B^{\dagger}(A-(J-R)P^{-1})$ stabilizes the triplet $(A,B,I_n)$ hence the name `ABI'. 

\item AIC. This is similar to ABI except that we choose $P$ such that there exists $J$ and $R$ where $K = (A-(J-R)P^{-1})C^{\dagger}$ stabilizes the triplet $(A,I_n,C)$ hence the name `AIC'.   

\end{enumerate}

\subsection{Extension to $\Omega$-stabilization}

Our approach to tackle the SOF and SSF problems can be extended to find feedbacks that make the system $\Omega$-stable. 
A matrix is said to be $\Omega$-stable if its eigenvalues belong to the set $\Omega \subset \mathbb{C}$.  
In~\cite{ChoGS19a}, a parametrization of $\Omega$-stable matrices was obtained
in terms of DH matrices that satisfy additional LMI constraints, where $\Omega$ is the intersection of specific regions in the complex plane; namely conic sectors, vertical strips, and disks. 
For example, one may want the real parts of the eigenvalues of $A-BKC$ to be strictly smaller than some given negative value so that $A-BK$ is robustly stable (see also Remark~\ref{robuststab} below). 
Using these parametrizations, the results obtained in Sections~\ref{sec:pair} and~\ref{sec:triplet} can
be directly extended to obtain a characterization of static stabilizing feedbacks that guarantee 
$\Omega$-stability in terms of DH matrices. 
The corresponding
optimization problems~\eqref{eq:algo2} and~\eqref{eq:tripletinit} 
would be subject to additional LMIs on the variables $J$, $R$ and $Q$
depending on the $\Omega$ region; see~\cite[Theorems 1-3]{ChoGS19a}.

\begin{remark} \label{robuststab} {\rm
Note that with the current version of the code, 
replacing $A$ with $A + \rho I$ allows to compute a feedback matrix that makes the real parts of the eigenvalues of $A-BKC$ smaller than $-\rho$. 
In fact, the real parts of the eigenvalues of $A-BKC$ are smaller than $-\rho$ if and only if 
the real parts of the eigenvalues of $A+\rho I-BKC$ are smaller than 0.} 
\end{remark}

\section{Numerical experiments} \label{sec:numexp}

In this section, we run our algorithms on the data sets from the library~\cite{leibfritz2004compleib}. 
To solve the convex optimization subproblems involving LMIs, we use the interior point method SDPT3 (version 4.0)~\cite{toh1999sdpt3,tutuncu2003solving} with CVX as a modeling system~\cite{cvx,gb08}. 
Our code is available from \url{https://sites.google.com/site/nicolasgillis/code} and the numerical examples presented below can be directly run from this online code.
All tests are preformed using Matlab
R2015a on a laptop Intel CORE i7-7500U CPU @2.7GHz 24Go RAM. 
It has to be noted that our algorithms are very flexible when it comes to the choice of the norm to be minimized. In fact, CVX can essentially handle any norm. However, we present results for the $\ell_2$ norm which is standard in the literature.

We focus in this section on the SOF problem, which is more challenging. Numerical results for the stabilization of matrix pairs $(A,B)$ are reported in Appendix~\ref{appA}. 
We compare Algorithm~\ref{algo:stabtriplet} with the solutions computed by two state-of-the-art algorithms: 
\begin{itemize}

\item The algorithm proposed in~\cite{Per16} is a randomized approximation algorithm. It works in to phases, similarly as Algorithm~\ref{algo:stabtriplet}. 
In the first phase (RS-PHASE-I, RS stands for Ray-Shooting), the algorithm looks for a feasible solution and, in the second phase (RS-PHASE-II), it improves the solution by minimizing its $\ell_2$-norm while remaining feasible. 
Initially, we tried to reproduce the results in~\cite{Per16} but this is impossible due to the randomized part of the algorithm. 
Moreover, there are several parameters of the algorithm that need to be fine-tuned to obtain good solutions, and we were not able to produce solutions as good as those presented in~\cite{Per16}. This is the reason why we prefer to report the results from~\cite{Per16}. 
Unfortunately, 
the author only provided extensive numerical results for the RS-PHASE-I of his algorithm~\cite[Table 8]{Per16}, and only 4 solutions obtained after RS-PHASE-II. 
Hence, we only compare to these solutions in this paper. 


\item The algorithm proposed in
\cite{H2hifoo} uses the HIFOO non-linear optimization Toolbox; see also~\cite{gumussoy2009multiobjective}.  
The algorithm, which we will refer to as HIFOO, relies on random initialization so, for the same reason as above, it is not possible to reproduce their results. 
  
\end{itemize}

For Algorithm~\ref{algo:stabtriplet}, we will use the four initializations presented in Section~\ref{init}. 
For random initialization, 
we report the best result out of 10 initializations.  
We use the $\ell_2$ norm as in~\cite{Per16, H2hifoo}, 
and the parameters $\underline{\epsilon} = \delta = 10^{-9}$; while when $\epsilon$ is decreased (resp.\@ increased), it is divided  (resp.\@ multiplied) by two (steps 8 and 12 of Algorithm~\ref{algo:stabtriplet}).  For both the initialization and optimization phases, we limit the number of iterations to 100. 
The initialization phase is stopped if the error is below $10^{-9}$. 
The optimization phase is stopped if the error is not decreased by at least $10^{-4}$ between two iterations. \\

For the stabilization of matrix pairs, we report the solutions found by Algorithm~\ref{algo:stabpair} in the  Appendix~\ref{appA}, 
and all the results can be rerun from our code available online.

\subsection{The four examples from \cite[Section 5.1]{Per16}}

Let us first compare our algorithm with 
RS-PHASE-II on the examples presented in~\cite[Section 5.1]{Per16}; see Table~\ref{tab:statoutres1}.  
The solution reported for Algorithm~\ref{algo:stabtriplet} is the best one obtained out of the four initializations (namely, identity for AC7 and AC8, random for HE1, and ABI for ROC7); 
see Tables~\ref{tab:statoutres} and~\ref{tab:statoutres2} 
for the results of the other initializations. 
\begin{center}  
 \begin{table}[h!] 
 \begin{center}
\caption{Comparison of the stabilizing feedback matrices $K$, 
their norm ${\|K\|}_2$ (the lowest norm is highlighted in bold), $\max_i \real{\lambda_i(A-BKC)}$, and the computational time to obtain them.} \label{tab:statoutres1} 
 \begin{tabular}{|c||c|c|} 
 \hline   
   & Algorithm~\ref{algo:stabtriplet} & RS-PHASE-II \\ \hline  
 AC7 & $K$ = [-0.2638   -0.2506] & $K$ = [-0.5963 -0.0669]\\ 
     &    ${\|K\|}_2=    \mathbf{0.36}$         & ${\|K\|}_2=0.60$ \\ 
     &    $\max_i \real{ \lambda_i(A-BKC)}= -0.0004$     & 
     $\max_i \real{ \lambda_i(A-BKC)}= -0.0056$  \\ 
& 7.4 s. &    0.17 s.   
     \\  \hline  
 AC8 & $K$ = [-0.0112    0.0049    0.0129   -0.0008   -0.0128] & $K$ = [-2.0783 0.2537 0.9202 -0.0382 -0.8338]\\ 
     &        ${\|K\|}_2 = \mathbf{0.02}$         & ${\|K\|}_2=  2.43$ \\ 
     &        $\max_i \real{ \lambda_i(A-BKC)}=-0.0257$     & $\max_i \real{ \lambda_i(A-BKC)}=-0.186$ \\ 
& 7.5 s. &   0.03 s.   \\   \hline 
HE1 
& $K$ = [0.1221;   -0.3974] 
& $K$ = [-0.0691;   -0.4227]\\ 
     &    ${\|K\|}_2 =  0.42$         
     & ${\|K\|}_2 = \mathbf{0.35}$ \\ 
     &    $\max_i \real{ \lambda_i(A-BKC)}= -0.0435$     & $\max_i \real{ \lambda_i(A-BKC)}= -0.0196$ \\ 
& 61 s. &   0.03 s.    \\  \hline     
     ROC7 & 
$K = 10^{-5}\left[ \begin{array}{ccc}      
     0.0148 &  \text{-}0.0016  & \text{-}0.0123\\ 
    0.6947  &  0.0113  &  0.2792  \\  
    \end{array} \right]$
& $K = \left[ \begin{array}{ccc} 
0.5659 & 0.0379 & \text{-}0.1363 \\
 \text{-}1.3274 & \text{-}0.9144 & 0.4762 \end{array} \right]$ \\ 
     &    ${\|K\|}_2 =  \mathbf{7.5 \, 10^{-6}}$         
     & ${\|K\|}_2 = 1.76$ \\ 
     &    $\max_i \real{ \lambda_i(A-BKC)} = 1.5 \, 10^{-7}$     
     & $\max_i \real{ \lambda_i(A-BKC)}= -0.0238$ \\ 
& 13 s. &   0.09 s.   
\\  \hline 
\end{tabular} 
 \end{center} 
 \end{table} 
 \end{center}

For these four numerical experiments, Algorithm~\ref{algo:stabtriplet} provides better solutions than RS-PHASE-II in 3 out of the 4 cases. 
In particular, for AC8 and ROC7, the improvement is significant; from ${\|K\|}_2 = 2.43$ to  ${\|K\|}_2 = 0.02$ for AC8, and 
from ${\|K\|}_2 = 1.76$ to ${\|K\|}_2 = 7.5 \, 10^{-6}$ for\footnote{
Note that the real part of one eigenvalue is positive, namely $1.5 \, 10^{-7}$. 
One way to ensure that this does not happen is to add $\rho I$ to the input matrix; see Remark~\ref{robuststab}. 
Another way is to use a lower bound $\delta$ on the eigenvalues of $R$ and $Q$ which can be easily done in our code. 
For example, using $\delta = 10^{-6}$, we obtain 
$K = 10^{-3} [1.4 \, 0  \, 0; 0  \, 0  \, 0.11]$ with $||K||_2 = 1.4 \, 10^{-3}$ 
with  $\max_i \real{ \lambda_i(A-BKC)} = -1.3 \, 10^{-17}$. 
} 
ROC7. 
When RS-PHASE-II obtains the best solution (HE1), 
the difference in norm is not significant: 0.42 for vs.\@ 0.35. In summary, RE-PHASE-II is significantly faster than Algorithm~\ref{algo:stabtriplet}, but generates in general solutions with larger norm. 
These observations will be confirmed in the next section.

 \subsection{Extensive numerical results}

In this section, we report numerical results for all the systems tested 
in~\cite[Table 8]{Per16} 
and~\cite[Table 1]{H2hifoo} of size $n \leq 20$; 
see 
 Tables~\ref{tab:statoutres}  and~\ref{tab:statoutres2}. 
Note that some systems in~\cite[Table 8]{Per16}  were not tested in~\cite[Table 1]{H2hifoo}, and vice versa, in which case we do not report any solution. 
Moreover, \cite[Table 8]{Per16} only reports the error of the initialization phase (RS-PHASE-I) hence we only report this result. 
Since~\cite[Table 8]{Per16} 
and~\cite[Table 1]{H2hifoo} minimize the $\ell_2$ norm of $K$, we run Algorithm~\ref{algo:stabtriplet} using this norm.


In terms of solution quality, Algorithm~\ref{algo:stabtriplet} outperforms the two other approaches, providing the best solution in 40 out of the 57 cases (RS-PHASE-I does for 4 out of 50, HIFOO does for 11 out of 45). Moreover, in many cases, the $\ell_2$ norm of the stabilizing feedback matrix is much smaller.  

In terms of finding feasible solutions, the three algorithms are comparable: Algorithm~\ref{algo:stabtriplet} only fails 5 times out of 57, 
RS-PHASE-I 4 times out of 50, and 
HIFOO 4 times out of 45. 
All algorithms are sometimes able to find a feasible solution while the other fail. For example, RS-PHASE-I fails on AC18 and WEC1 while the two other algorithms succeed; 
HIFOO fails on HE6, HE7, NN13 and NN14; and Algorithm~\ref{algo:stabtriplet} fails on NN3-6-7-9 and ROC2-3. 

Note that no algorithm is able to return a solution for NN3 and ROC3 hence these systems might not be stabilizable.

\paragraph{Sensitivity to initialization} As expected, Algorithm~\ref{algo:stabtriplet} is rather sensitive to initialization. In terms of solution quality, ABI performs best (23 out of 57 best solution found). In terms of finding feasible solutions, Random performs best (only 5 failure out of 57) but used 10 initializations. 
In fact, we were curious to see whether the random initialization would be able to find feasible solutions if allowed more trials. We have rerun the same experiment with 100 initializations, and it was able to find a feasible solution for NN7 with $||K||_2=98.45$, and for ROC2 with $||K||_2 = 3.31$; both being better solutions than the ones found by RS-PHASE-I and HIFOO.  Hence, except for NN6, Algorithm~\ref{algo:stabtriplet} with random initialization was able to find feasible solutions whenever RS-PHASE-I and HIFOO did.

\paragraph{Computational time}

RS-PHASE-I and -II require less than 1 seconds in all the examples shown on
Tables~\ref{tab:statoutres} and~\ref{tab:statoutres2}.   
Although it is not reported here, in terms of computational time, HIFOO is in general faster than Algorithm~\ref{algo:stabtriplet} but significantly slower than RS-PHASE-I and -II. 
 For example, on NN9 ($n = 5$, $m = 3$ and $p = 2$), 
 HIFOO requires about 6 seconds, 
 while Algorithm~\ref{algo:stabtriplet} requires about 40 seconds on average. 
 Note that the random initialization is slower than the other ones as it runs Algorithm~\ref{algo:stabtriplet} 10 times.

\begin{center}  
 \begin{table}[h!] 
 \begin{center} 
\caption{
Norms ${\|K\|}_2$ of the feedback matrices $K$  (and, in brackets, the computational time in seconds) generated by Algorithm~\ref{algo:stabtriplet} using different initializations (second to fifth column) on problems from the COMPLeIB library. 
The last two columns report the results of the algorithm of 
Peretz (RS-PHASE-I=RS-P-I)~\cite[Table~8]{Per16} 
and 
H2-HIFOO~\cite[Table~1]{H2hifoo}. 
An empty box means that the result is not available, 
$\infty$ means that the algorithm did not find a feasible feedback, 
e-x means $10^{-x}$. 
The solutions with error at most 0.01 away from the best solution found are highlighted in bold.  
\label{tab:statoutres} 
} 

 \begin{tabular}{|c||cccc|c|c|} 
 \hline  
   & \multicolumn{4}{c|}{Algorithm~\ref{algo:stabtriplet}}   & RS-P.-I & HIFOO 
   \\ 
   &  Identity & Random & ABI  & AIC  & \cite{Per16}  &  \cite{H2hifoo}
 \\ \hline 
 AC1 &   \textbf{1.63e-9} (1.5)   &  \textbf{3.65e-9} (2.9)   &  \textbf{1.63e-9} (1.1)   &  \textbf{1.63e-9} (1.1)   &  1.62   &  \textbf{1.81e-15}    \\ \hline 
 AC2 &   \textbf{1.63e-9} (1.3)   &  \textbf{3.65e-9} (3.3)   &  \textbf{1.63e-9} (1.1)   &  \textbf{1.63e-9} (1.1)   &  1.62   &  4.91e-2    \\ \hline 
 AC4 &   7.48e-2 (2.1)   &  6.69e-2 (34)   &  7.27e-2 (8.0)   &  \textbf{5.21e-2} (15)   &  0.77   &  \\ \hline 
 AC5 &   1358.60 (16)   &  1344.56 (718)   &  2282.84 (28)   &  $\infty$ (21)   &  \textbf{1.41}   &  1340.00   \\ \hline 
 AC6 &   \textbf{4.05e-12} (1.1)   &  \textbf{4.05e-12} (3.3)   &  \textbf{4.05e-12} (1.1)   &  \textbf{4.05e-12} (1.1)   &  1.20   &  \\ \hline 
 AC7 &   \textbf{0.36} (6.5)   &  0.42 (68)   &  0.76 (21)   &  0.41 (59)   &  1.02   &  \\ \hline 
 AC8 &   \textbf{2.19e-2} (6.6)   &  0.12 (142)   &  1.20 (19)   &  0.50 (34)   &  2.43   &  \\ \hline 
 AC9 &   0.66 (7.9)   &  2.34 (137)   &  \textbf{1.31e-3} (36)   &  0.26 (44)   &  1.92   &  1.41   \\ \hline 
 AC11 &   5.98 (11)   &  \textbf{0.17} (510)   &  1.30 (20)   &  0.92 (56)   &  2.41   &  3.64   \\ \hline 
 AC12 &   1.79 (71)   &  1.19 (540)   &  1.00 (28)   &  0.90 (35)   &  3398.70   &  \textbf{5.00e-5}    \\ \hline 
 AC15 &   \textbf{6.34e-10} (0.9)   &  \textbf{6.34e-10} (2.6)   &  \textbf{6.34e-10} (0.9)   &  \textbf{6.34e-10} (0.9)   &  8.71   &  \\ \hline 
 AC18 &   1.60 (223)   &  \textbf{0.37} (1688)   &  0.55 (46)   &  1.26 (173)   &  $\infty$    &  18.60   \\ \hline 
 HE1 &   2.26 (1.8)   &  0.42 (61)   &  0.68 (51)   &  0.46 (71)   &  1.67   &  \textbf{8.57e-2}    \\ \hline 
 HE3 &   49.77 (26)   &  29.12 (589)   &  9.50 (136)   &  10.75 (140)   &  41.91   &  \textbf{0.81}   \\ \hline 
 HE4 &   45.65 (22)   &  13.01 (569)   &  \textbf{5.68} (155)   &  9.42 (127)   &  145.17   &  18.60   \\ \hline 
 HE5 &   $\infty$ (38)   &  33.39 (1139)   &  $\infty$ (180)   &  16.60 (38)   &  144.95   &  \textbf{1.59}   \\ \hline 
 HE6 &   4.41 (15)   &  \textbf{2.30} (660)   &  2.47 (54)   &  9.83 (96)   &  3.63   &  $\infty$    \\ \hline 
 HE7 &   4.41 (15)   &  3.36 (439)   &  \textbf{2.47} (55)   &  9.83 (95)   &  3.63   &  $\infty$    \\ \hline 
 REA1 &   1.06 (2.6)   &  \textbf{0.85} (69)   &  0.94 (18)   &  1.79 (18)   &  1.55   &  1.50   \\ \hline 
 REA2 &   1.73 (3.3)   &  \textbf{0.92} (45)   &  1.09 (15)   &  2.62 (16)   &  1.12   &  1.65   \\ \hline 
 REA3 &   \textbf{2.40e-3} (0.0)   &  \textbf{2.40e-3} (30)   &  \textbf{2.40e-3} (0.0)   &  \textbf{2.40e-3} (0.0)   &  14.84   &  9.91   \\ \hline 
 DIS2 &   6.18 (3.7)   &  4.94 (47)   &  8.59 (9.6)   &  3.24 (9.2)   &  &  \textbf{1.40}   \\ \hline 
 DIS4 &   0.44 (70)   &  0.36 (227)   &  \textbf{0.32} (24)   &  0.40 (24)   &  2.82   &  1.69   \\ \hline 
 DIS5 &   $\infty$ (32)   &  \textbf{416.81} (503)   &  $\infty$ (63)   &  $\infty$ (30)   &  461.97   &  1280.00   \\ \hline 
 WEC1 &   \textbf{4.30} (6.2)   &  25.98 (1365)   &  $\infty$ (233)   &  9.05 (37)   &  $\infty$    &  5.69   \\ \hline 
 PAS &   4.72e-3 (0.0)   &  4.72e-3 (13)   &  4.72e-3 (0.0)   &  4.72e-3 (0.0)   &  780.00   &  \textbf{1.97e-3}    \\ \hline 
 TF1 &   1.16 (8.8)   &  4.00 (100)   &  0.21 (19)   &  $\infty$ (126)   &  65.04   &  \textbf{0.14}   \\ \hline 
 TF2 &   5.40 (3.2)   &  6.31 (154)   &  \textbf{0.42} (19)   &  1.69 (16)   &  7.95   &  10.90   \\ \hline 
 TF3 &   $\infty$ (97)   &  47.13 (656)   &  0.28 (19)   &  1.64 (31)   &  151.57   &  \textbf{0.14}   \\ \hline 
 NN1 &   $\infty$ (16)   &  48.43 (177)   &  $\infty$ (43)   &  $\infty$ (39)   &  133.69   &  \textbf{35.00}   \\ \hline 
 NN2 &   \textbf{3.45e-9} (1.2)   &  \textbf{3.45e-9} (2.8)   &  \textbf{3.45e-9} (1.1)   &  \textbf{3.45e-9} (1.1)   &  1.35   &  1.54   \\ \hline 
 NN3 &   $\infty$ (31)   &  $\infty$ (264)   &  $\infty$ (37)   &  $\infty$ (12)   &  $\infty$    &  \\ \hline 
 NN5 &   $\infty$ (30)   &  \textbf{17.14} (556)   &  $\infty$ (73)   &  $\infty$ (114)   &  39.03   &  82.40   \\ \hline 
 NN6 &   $\infty$ (115)   &  $\infty$ (1446)   &  $\infty$ (188)   &  $\infty$ (264)   &  \textbf{110.73}   &  314.00   \\ \hline 
\end{tabular} 
 \end{center} 
 \end{table} 
 \end{center} 
 
 \begin{center}  
 \begin{table}[h!] 
 \begin{center} 
 \caption{ 
Continued from Table~\ref{tab:statoutres}. \label{tab:statoutres2}
} 
 \begin{tabular}{|c||cccc|c|c|} 
 \hline  
   & \multicolumn{4}{c|}{Algorithm~\ref{algo:stabtriplet}}   & RS-P.-I & HIFOO 
   \\ 
   &  Identity & Random & ABI  & AIC  & \cite{Per16}  &  \cite{H2hifoo}
 \\ \hline 
  NN7 &   $\infty$ (118)   &  $\infty$ (1243)   &  $\infty$ (187)   &  $\infty$ (264)   &  \textbf{71.53}   &  84.20   \\ \hline 
 NN9 &   $\infty$ (39)   &  \textbf{10.66} (250)   &  $\infty$ (41)   &  $\infty$ (65)   &  504.50   &  20.90   \\ \hline 
 NN12 &   32.67 (25)   &  18.23 (327)   &  $\infty$ (88)   &  $\infty$ (33)   &  27.05   &  \textbf{10.90}   \\ \hline 
 NN13 &   0.61 (2.7)   &  0.40 (47)   &  \textbf{7.05e-2} (29)   &  9.98e-2 (17)   &  1.94   &  $\infty$    \\ \hline 
 NN14 &   0.61 (2.6)   &  0.38 (51)   &  \textbf{7.05e-2} (29)   &  9.98e-2 (17)   &  1.47   &  $\infty$    \\ \hline 
 NN15 &   \textbf{4.71e-11} (1.0)   &  \textbf{4.70e-11} (2.7)   &  \textbf{4.71e-11} (1.0)   &  \textbf{4.71e-11} (1.0)   &  2.00   &  4.80e-2    \\ \hline 
 NN16 &   \textbf{1.52e-10} (1.1)   &  \textbf{1.83e-10} (3.1)   &  \textbf{1.52e-10} (1.0)   &  \textbf{1.52e-10} (1.0)   &  0.46   &  0.34   \\ \hline 
 NN17 &   $\infty$ (7.6)   &  53.23 (323)   &  1.86 (14)   &  \textbf{1.41} (11)   &  6.77   &  3.87   \\ \hline 
 HF2D10 &   1.23 (2.1)   &  0.30 (31)   &  0.31 (3.7)   &  \textbf{0.29} (13)   &  15.41   &  70600    \\ \hline 
 HF2D11 &   8.36 (2.1)   &  1.64 (32)   &  0.63 (3.6)   &  \textbf{0.58} (7.1)   &  44.02   &  85100    \\ \hline 
 HF2D14 &   7.48e-2 (18)   &  2.37e-2 (97)   &  \textbf{2.06e-2} (4.8)   &  \textbf{2.04e-2} (9.8)   &  &  373000    \\ \hline 
 HF2D15 &   0.91 (14)   &  0.26 (258)   &  \textbf{0.26} (14)   &  1.34 (103)   &  &  284000    \\ \hline 
 HF2D16 &   2.98e-2 (3.0)   &  1.49e-2 (93)   &  1.57e-2 (4.4)   &  \textbf{1.39e-2} (4.2)   &  &  284000    \\ \hline 
 HF2D17 &   0.10 (2.3)   &  \textbf{6.63e-2} (125)   &  6.81e-2 (12)   &  \textbf{6.59e-2} (11)   &  &  375000    \\ \hline 
 HF2D18 &   2.45e-2 (2.1)   &  1.49e-2 (25)   &  1.84e-2 (7.1)   &  \textbf{7.00e-3} (15)   &  &  24.30   \\ \hline 
 TMD &   0.28 (2.8)   &  0.30 (88)   &  \textbf{2.15e-3} (7.7)   &  119.22 (37)   &  1.07   &  1.32   \\ \hline 
 FS &   $\infty$ (35)   &  793.18 (382)   &  \textbf{145.34} (20)   &  1.26e+04 (38)   &  &  18300    \\ \hline 
 ROC1 &   \textbf{5.63e-6} (0.0)   &  \textbf{5.63e-6} (11)   &  \textbf{5.63e-6} (0.0)   &  \textbf{5.63e-6} (0.0)   &  180.14   &  \\ \hline 
 ROC2 &   $\infty$ (54)   &  $\infty$ (1471)   &  $\infty$ (105)   &  $\infty$ (185)   &  \textbf{152.94}   &  \\ \hline 
 ROC3 &   $\infty$ (41)   &  $\infty$ (642)   &  $\infty$ (72)   &  $\infty$ (107)   &  $\infty$    &  \\ \hline 
 ROC4 &   \textbf{2.58e-6} (1.2)   &  \textbf{2.58e-6} (22)   &  \textbf{2.58e-6} (1.1)   &  \textbf{2.58e-6} (1.1)   &  241.57   &  \\ \hline 
 ROC5 &   \textbf{2.39e-9} (1.1)   &  \textbf{1.33e-9} (3.1)   &  \textbf{2.39e-9} (1.1)   &  \textbf{2.39e-9} (1.1)   &  232.22   &  \\ \hline 
 ROC7 &   0.18 (3.3)   &  0.37 (98)   &  \textbf{7.49e-6} (9.2)   &  3.01e-3 (15)   &  2.32   &  \\ \hline 
 \hline  \# best  & 14 / 57 & 20 / 57 & 23 / 57 & 19 / 57 &  4 / 50 & 11 / 45 
 \\ \hline 
  &  
\multicolumn{4}{c|}{Globally: 40 / 57}
  &  &  
 \\ \hline 
 \# $\infty$ & 13 / 57 &  5 / 57 & 12 / 57 & 12 / 57 &  4 / 50 &  4 / 45 
 \\ \hline 
  &  
  \multicolumn{4}{c|}{Globally: 5 / 57}
  &  &  
 \\ \hline 
\end{tabular} 
 \end{center} 
 \end{table} 
 \end{center}

\newpage

\section{Conclusion} 

In this paper, we have proposed a new characterization of 
all the SSFs and SOFs of a given LTI system pair $(A,B)$ 
and system triplet $(A,B,C)$, respectively, in terms of DH matrices.
This allowed us to develop algorithms to compute  minimal-norm
SSFs for a system pair $(A,B)$ (Algorithm~\ref{algo:stabpair}) and minimal-norm SOFs for a system triplet $(A,B,C)$ (Algorithm~\ref{algo:stabtriplet}).
Comparing Algorithm~\ref{algo:stabtriplet} with the methods 
HIFOO~\cite{H2hifoo} and RS~\cite{Per16} on SOF problems, we found that RS performs better than  Algorithm~\ref{algo:stabtriplet} in terms of computational time. 
In terms of solution quality,  Algorithm~\ref{algo:stabtriplet} compares favourably with the two other methods, being able to obtain better solution in many cases. 
In terms of finding feasible solutions, the three methods perform similarly.  

Further work include the design of faster algorithms to solve our SDPs such as first-order methods. In fact, our algorithms  currently  rely on interior-point methods which do not scale well. 


\section*{Acknowledgements}

The authors are grateful to Yossi Peretz for sharing his code, and adressing our questions regarding his paper~\cite{Per16}.



\small

\bibliographystyle{spmpsci}
\bibliography{GilS}

\begin{thebibliography}{10}
\providecommand{\url}[1]{{#1}}
\providecommand{\urlprefix}{URL }
\expandafter\ifx\csname urlstyle\endcsname\relax
  \providecommand{\doi}[1]{DOI~\discretionary{}{}{}#1}\else
  \providecommand{\doi}{DOI~\discretionary{}{}{}\begingroup
  \urlstyle{rm}\Url}\fi

\bibitem{Alb69}
Albert, A.: Conditions for positive and nonnegative definiteness in terms of
  pseudoinverses.
\newblock SIAM J. Appl. Math. \textbf{17}(2), 434--440 (1969)

\bibitem{AndBJ75}
Anderson, B., Bose, N., Jury, E.: Output feedback stabilization and related
  problems-- solution via decision methods.
\newblock IEEE Transactions on Automatic Control \textbf{20}(1), 53--66 (1975)

\bibitem{H2hifoo}
Arzelier, D., Deaconu, G., Gumussoy, S., Henrion, D.: H2 for {HIFOO}.
\newblock In: International Conference on Control and Optimization with
  Industrial Applications, Bilkent University, Ankara, Turkey (2011)

\bibitem{AstM10}
{\AA}str\"{o}m, K.J., Murray, R.M.: Feedback Systems: An Introduction for
  Scientists and Engineers.
\newblock Princeton University Press (2010)

\bibitem{BeaMV18}
Beattie, C.A., Mehrmann, V., Van~Dooren, P.: Robust port-{H}amiltonian
  representations of passive systems.
\newblock Automatica \textbf{100}, 182--186 (2019)

\bibitem{BloT97}
Blondel, V., Tsitsiklis, J.: {NP}-{H}ardness of some linear control design
  problems.
\newblock SIAM Journal on Control and Optimization \textbf{35}(6), 2118--2127
  (1997)

\bibitem{ChoGS19a}
Choudhary, N., Gillis, N., Sharma, P.: On approximating the nearest
  {$\Omega$}-stable matrix.
\newblock arXiv preprint arXiv:1901.03069  (2019)

\bibitem{cvx}
CVX~Research, I.: {CVX}: Matlab software for disciplined convex programming,
  version 2.0.
\newblock \url{http://cvxr.com/cvx} (2012)

\bibitem{GilKS18a}
Gillis, N., Karow, M., Sharma, P.: Approximating the nearest stable
  discrete-time system.
\newblock Linear Algebra and its Applications \textbf{573}, 37--53 (2019)

\bibitem{GilMS17}
Gillis, N., Mehrmann, V., Sharma, P.: Computing nearest stable matrix pairs.
\newblock Numerical Linear Algebra with Applications pp. e2153, (2018).
\newblock {doi:10.1002/nla.2153}

\bibitem{GilS17}
Gillis, N., Sharma, P.: On computing the distance to stability for matrices
  using linear dissipative {H}amiltonian systems.
\newblock Automatica \textbf{85}, 113--121 (2017)

\bibitem{GilS17b}
Gillis, N., Sharma, P.: Finding the nearest positive-real system.
\newblock SIAM Journal on Numerical Analysis \textbf{56}(2), 1022--1047 (2018)

\bibitem{GolSBM03}
Golo, G., Schaft, A.v., Breedveld, P., Maschke, B.: {H}amiltonian formulation
  of bond graphs.
\newblock In: A.R. R.~Johansson (ed.) Nonlinear and Hybrid Systems in
  Automotive Control, pp. 351--372. Springer-Verlag, Heidelberg, Germany (2003)

\bibitem{gb08}
Grant, M., Boyd, S.: Graph implementations for nonsmooth convex programs.
\newblock Recent Advances in Learning and Control pp. 95--110 (2008)

\bibitem{gumussoy2009multiobjective}
Gumussoy, S., Henrion, D., Millstone, M., Overton, M.L.: Multiobjective robust
  control with {HIFOO} 2.0.
\newblock IFAC Proceedings Volumes \textbf{42}(6), 144--149 (2009)

\bibitem{leibfritz2004compleib}
Leibfritz, F.: Compleib, constraint matrix-optimization problem library-a
  collection of test examples for nonlinear semidefinite programs, control
  system design and related problems.
\newblock Dept. Math., Univ. Trier, Trier, Germany, Tech. Rep  (2004)

\bibitem{MehMS16}
Mehl, C., Mehrmann, C., Sharma, P.: Stability radii for linear {H}amiltonian
  systems with dissipation under structure-preserving perturbations.
\newblock SIAM J. Matrix Anal. Appl. \textbf{37}(4), 1625--1654 (2016)

\bibitem{MehV19}
Mehrmann, V., Van~Dooren, P.: Optimal robustness of port-{H}amiltonian systems.
\newblock arXiv preprint arXiv:1904.13326  (2019)

\bibitem{Nem93}
Nemirovskii, A.: Several {NP}-hard problems arising in robust stability
  analysis.
\newblock Mathematics of Control, Signals and Systems \textbf{6}(2), 99--105
  (1993)

\bibitem{Per12}
Peretz, Y.: A characterization of all the static stabilizing controllers for
  {LTI} systems.
\newblock Linear Algebra and its Applications \textbf{437}(2), 525 -- 548
  (2012)

\bibitem{Per16}
Peretz, Y.: A randomized approximation algorithm for the minimal-norm
  static-output-feedback problem.
\newblock Automatica \textbf{63}, 221 -- 234 (2016)

\bibitem{PolKS03}
Polyak, B., Khlebnikov, M., Shcherbakov, P.: An {LMI} approach to structured
  sparse feedback design in linear control systems.
\newblock In: 2013 European Control Conference (ECC), pp. 833--838 (2013)

\bibitem{Sch06}
Schaft, A.v.: Port-{H}amiltonian systems: an introductory survey.
\newblock In: J.V. M.~Sanz-Sole, J.~Verdura (eds.) Proc. of the International
  Congress of Mathematicians, vol. III, Invited Lectures, pp. 1339ñ--1365.
  Madrid, Spain

\bibitem{SchM13}
Schaft, A.v., Maschke, B.: Port-{H}amiltonian systems on graphs.
\newblock SIAM J. Control Optim.  (2013)

\bibitem{SpeS97}
Spencer, B.F., Sain, M.K.: Controlling buildings: a new frontier in feedback.
\newblock IEEE Control Systems Magazine \textbf{17}(6), 19--35 (1997)

\bibitem{Sun93}
Sun, J.G.: Backward perturbation analysis of certain characteristic subspaces.
\newblock Numerische Mathematik \textbf{65}, 357--382 (1993)

\bibitem{SyrADG97}
Syrmos, V.L., Abdallah, C.T., Dorato, P., Grigoriadis, K.: Static output
  feedback: A survey.
\newblock Automatica \textbf{33}(2), 125 -- 137 (1997)

\bibitem{toh1999sdpt3}
Toh, K.C., Todd, M., T{\"u}t{\"u}nc{\"u}, R.: {SDPT3}--a {MATLAB} software
  package for semidefinite programming, version 1.3.
\newblock Optimization Methods and Software \textbf{11}(1-4), 545--581 (1999)

\bibitem{TreSAH12}
Trentelman, H.L., Stoorvogel, A.A., Hautus, M.: Control theory for linear
  systems.
\newblock Springer Science \& Business Media (2012)

\bibitem{tutuncu2003solving}
T{\"u}t{\"u}nc{\"u}, R., Toh, K., Todd, M.: Solving
  semidefinite-quadratic-linear programs using {SDPT3}.
\newblock Math. Program. \textbf{95}(2), 189--217 (2003)

\bibitem{XuT02}
Xu, Y., Teng, J.: Optimum design of active/passive control devices for tall
  buildings under earthquake excitation.
\newblock The Structural Design of Tall Buildings \textbf{11}(2), 109--127
  (2002)

\bibitem{YanLJ03}
Yang, J.N., Lin, S., Jabbari, F.: H2-based control strategies for civil
  engineering structures.
\newblock Journal of Structural Control \textbf{10}(3-4), 205--230 (2003)

\end{thebibliography}

\normalsize 

\newpage 

\appendix 

\section{Stabilizing matrix pairs $(A,B)$} \label{appA}

In this section, we report the $\ell_2$ norm of the solutions obtained by Algorithm~\ref{algo:stabpair} for the SSF problems corresponding to the same instance as in 
Tables~\ref{tab:statoutres} 
and~\ref{tab:statoutres2}.   
We minimize the $\ell_2$ norm of the feedback matrices, use $\underline{\epsilon} = 10^{-9}$, and update $\epsilon$ in the same way as in Section~\ref{sec:numexp}.  
We also report the error of the $\ell_2$ norm of the solution obtained by solving~\eqref{eq:algo3ab} (initialization phase of Algorithm~\ref{algo:stabpair}), and the error of the BCD algorithm that alternatively optimized $(J,R)$ for $Q$ fixed, and vice versa. 

As mentioned in Section~\ref{sec:algorithm}, 
SSDP performs significantly better than BCD in terms of solution quality.  
BCD provides a slightly better solution only in a few cases.  
In terms of computational time, BCD is faster as it solves subproblems with fewer variables.

\begin{center}  
 \begin{table}[h!] 
 \begin{center} 
\caption{Comparison of the $\ell_2$ norm of the stabilizing feedback matrices (and, in brackets, the computational time in seconds and the number of iterations) 
of BCD and SSDP for the SSF problem. 
The solutions with error at most 0.01\% away from the best solution found are highlighted in bold, 
e-x means $10^{-x}$ and e+x means $10^{x}$. 
\label{tab:ssf1} } 
 \begin{tabular}{|c||ccc|} 
 \hline  
 Data set $(n,m)$ & Init. & BCD & SSDP   
 \\ \hline  AC1 (5,3) &  \textbf{2.82e-14} (1.0) &  \textbf{2.82e-14} (0.0,  0) &  \textbf{2.82e-14} (0.0,  0)   \\ \hline 
 AC2 (5,3) &  \textbf{2.82e-14} (0.3) &  \textbf{2.82e-14} (0.0,  0) &  \textbf{2.82e-14} (0.0,  0)   \\ \hline 
 AC4 (4,1) &  1.29  (1.2) &  3.86e-1 (2.4,  4) &  \textbf{7.91e-2} (10.4, 16)   \\ \hline 
 AC5 (4,2) &  1.43e+3 (0.8) &  3.83e+2 (15.1, 23) &  \textbf{2.85e+2} (280.4, 200)   \\ \hline 
 AC6 (7,2) &  \textbf{1.82e-15} (0.3) &  \textbf{1.82e-15} (0.0,  0) &  \textbf{1.82e-15} (0.0,  0)   \\ \hline 
 AC7 (9,1) &  2.33  (1.4) &  2.09e-1 (2.8,  4) &  \textbf{7.64e-2} (43.4, 25)   \\ \hline 
 AC8 (9,1) &  1.62  (1.5) &  3.30e-1 (2.5,  4) &  \textbf{3.79e-3} (37.5, 29)   \\ \hline 
 AC9 (10,4) &  1.13e-2 (1.7) &  \textbf{3.06e-4} (7.0,  4) &  4.02e-4 (53.9, 11)   \\ \hline 
 AC11 (5,2) &  2.44  (0.7) &  1.49  (21.2, 23) &  \textbf{7.44e-1} (19.2, 13)   \\ \hline 
 AC12 (4,3) &  4.73  (0.7) &  \textbf{2.89 } (6.3,  8) &  4.00  (233.0, 200)   \\ \hline 
 AC15 (4,2) &  \textbf{2.33e-13} (0.3) &  \textbf{2.33e-13} (0.0,  0) &  \textbf{2.33e-13} (0.0,  0)   \\ \hline 
 AC18 (10,2) &  2.28e-1 (9.0) &  3.21e-2 (23.4, 13) &  \textbf{1.15e-2} (85.1, 21)   \\ \hline 
 HE1 (4,2) &  1.62e-1 (0.4) &  1.26e-1 (3.3,  5) &  \textbf{1.18e-1} (12.8, 12)   \\ \hline 
 HE3 (8,4) &  3.05  (1.9) &  8.62e-1 (45.7, 35) &  \textbf{7.06e-1} (205.5, 88)   \\ \hline 
 HE4 (8,4) &  1.45  (1.6) &  3.28e-1 (65.3, 61) &  \textbf{3.87e-2} (39.2, 16)   \\ \hline 
 HE5 (8,4) &  1.45  (1.3) &  3.28e-1 (65.9, 61) &  \textbf{3.87e-2} (38.2, 16)   \\ \hline 
 HE6 (20,4) &  1.45  (6.8) &  7.83e-1 (553.4, 70) &  \textbf{3.57e-2} (316.4, 19)   \\ \hline 
 HE7 (20,4) &  1.45  (6.8) &  7.83e-1 (555.7, 70) &  \textbf{3.57e-2} (317.1, 19)   \\ \hline 
 REA1 (4,2) &  9.01e-1 (0.5) &  3.70e-1 (4.4,  7) &  \textbf{3.17e-1} (21.7, 20)   \\ \hline 
 REA2 (4,2) &  9.14e-1 (0.5) &  3.75e-1 (5.8,  7) &  \textbf{3.21e-1} (26.8, 25)   \\ \hline 
 REA3 (12,1) &  6.57e+2 (2.3) &  6.12e+2 (3.7,  4) &  \textbf{1.05e-2} (159.7, 96)   \\ \hline 
 DIS2 (3,2) &  2.60  (0.4) &  1.41  (7.6, 12) &  \textbf{1.22 } (21.4, 23)   \\ \hline 
 DIS4 (6,4) &  7.19e-1 (0.5) &  4.07e-1 (16.3, 21) &  \textbf{3.20e-1} (21.7, 16)   \\ \hline 
 DIS5 (4,2) &  7.83e+2 (0.5) &  1.58e+2 (1.3,  2) &  \textbf{1.03e+2} (63.8, 39)   \\ \hline 
 WEC1 (10,3) &  3.27e+1 (7.3) &  4.59  (261.5, 106) &  \textbf{3.97 } (72.3, 33)   \\ \hline 
 PAS (5,1) &  \textbf{6.82e-3} (0.9) &  \textbf{6.82e-3} (0.0,  0) &  \textbf{6.82e-3} (2.9,  1)   \\ \hline 
\end{tabular} 
 \end{center} 
 \end{table} 
 \end{center} 
 
 \begin{center}  
 \begin{table}[h!] 
 \begin{center} 
\caption{Continued from Table~\ref{tab:ssf1}. \label{tab:ssf2}} 
 \begin{tabular}{|c||ccc|} 
 \hline  
 Data set $(n,m)$ & Init. & BCD & SSDP   
 \\ \hline 
  TF1 (7,2) &  2.92e+1 (1.5) &  3.97  (2.8,  4) &  \textbf{7.77e-3} (30.2, 20)   \\ \hline 
 TF2 (7,2) &  2.92e+1 (1.6) &  3.97  (2.8,  4) &  \textbf{7.77e-3} (31.9, 20)   \\ \hline 
 TF3 (7,2) &  2.92e+1 (1.6) &  3.97  (2.8,  4) &  \textbf{7.77e-3} (30.4, 20)   \\ \hline 
 NN1 (3,1) &  1.65e+1 (0.8) &  1.34e+1 (2.4,  4) &  \textbf{1.30e+1} (22.9, 27)   \\ \hline 
 NN2 (2,1) &  \textbf{0} (0.1) &  \textbf{0} (0.0,  0) &  \textbf{0} (0.0,  0)   \\ \hline 
 NN3 (4,1) &  6.33e+1 (0.8) &  4.31e+1 (0.6,  1) &  \textbf{1.81e+1} (7.0,  6)   \\ \hline 
 NN5 (7,1) &  4.19e+2 (1.3) &  1.69e+2 (2.9,  4) &  \textbf{3.01e+1} (179.3, 200)   \\ \hline 
 NN6 (9,1) &  3.48e+2 (1.8) &  3.12e+2 (0.7,  1) &  \textbf{6.04e+1} (98.3, 60)   \\ \hline 
 NN7 (9,1) &  3.48e+2 (1.6) &  3.12e+2 (0.7,  1) &  \textbf{6.04e+1} (98.4, 60)   \\ \hline 
 NN9 (5,3) &  5.19  (0.6) &  4.50  (2.2,  3) &  \textbf{3.32 } (66.4, 64)   \\ \hline 
  NN12 (6,2) &  2.45  (0.8) &  2.45  (3.3,  4) &  \textbf{1.41 } (118.7, 111)   \\ \hline 
 NN13 (6,2) &  8.64e-1 (0.5) &  8.34e-1 (2.6,  4) &  \textbf{3.19e-1} (37.9, 35)   \\ \hline 
 NN14 (6,2) &  8.64e-1 (0.5) &  8.34e-1 (2.6,  4) &  \textbf{3.19e-1} (38.5, 35)   \\ \hline 
  NN15 (3,2) &  \textbf{2.22e-12} (0.3) &  \textbf{2.22e-12} (1.3,  2) &  \textbf{2.22e-12} (8.2,  1)   \\ \hline 
 NN16 (8,4) &  \textbf{5.40e-15} (0.3) &  \textbf{5.40e-15} (0.0,  0) &  \textbf{5.40e-15} (0.0,  0)   \\ \hline 
 NN17 (3,2) &  4.38  (0.7) &  8.30e-1 (5.1,  8) &  \textbf{7.26e-1} (10.2, 14)   \\ \hline 
 HF2D10 (5,2) &  9.37e-3 (0.5) &  \textbf{9.16e-3} (2.1,  3) &  {9.19e-3} (1.5,  3)   \\ \hline 
 HF2D11 (5,2) &  1.55e-2 (0.5) &  \textbf{1.50e-2} (2.9,  4) &  \textbf{1.50e-2} (1.3,  3)   \\ \hline 
 HF2D14 (5,2) &  2.97e-2 (0.5) &  \textbf{2.35e-2} (2.8,  4) &  {2.36e-2} (2.6,  5)   \\ \hline 
 HF2D15 (5,2) &  1.98e-1 (0.6) &  \textbf{1.22e-1} (8.1, 12) &  \textbf{1.22e-1} (6.0,  5)   \\ \hline 
 HF2D16 (5,2) &  1.46e-2 (0.5) &  \textbf{1.40e-2} (2.7,  4) &  \textbf{1.40e-2} (1.3,  3)   \\ \hline 
 HF2D17 (5,2) &  6.83e-2 (0.5) &  \textbf{6.47e-2} (2.9,  4) &  \textbf{6.47e-2} (2.4,  5)   \\ \hline 
 HF2D18 (5,2) &  3.79e-2 (0.5) &  \textbf{3.59e-2} (1.3,  2) &  3.79e-2 (5.0,  2)   \\ \hline 
 TMD (6,2) &  4.73e-3 (1.2) &  \textbf{1.81e-4} (2.8,  4) &  2.22e-4 (4.9,  8)   \\ \hline 
 FS (5,1) &  5.28e+2 (2.2) &  5.28e+2 (0.0,  0) &  \textbf{3.04e+1} (6.1,  8)   \\ \hline 
 ROC1 (9,2) &  \textbf{8.82e-5} (0.6) &  \textbf{8.82e-5} (1.3,  1) &  \textbf{8.82e-5} (35.6,  1)   \\ \hline 
 ROC2 (10,2) &  2.36  (2.9) &  2.93e-1 (5.1,  4) &  \textbf{7.78e-2} (121.4, 35)   \\ \hline 
 ROC3 (11,4) &  7.24  (1.7) &  6.68  (12.8,  7) &  \textbf{3.63 } (92.2, 19)   \\ \hline 
 ROC4 (9,2) &  \textbf{1.79e-4} (0.6) &  \textbf{1.79e-4} (2.2,  2) &  \textbf{1.79e-4} (26.8,  1)   \\ \hline 
 ROC5 (7,3) &  \textbf{4.57e-14} (0.3) &  \textbf{4.57e-14} (0.0,  0) &  \textbf{4.57e-14} (0.0,  0)   \\ \hline 
 ROC7 (5,2) &  2.20  (1.8) &  4.58e-1 (4.2,  7) &  \textbf{1.33e-4} (8.6, 13)   \\ \hline 
\end{tabular} 
 \end{center} 
 \end{table} 
 \end{center}

\end{document}